\documentclass[english,11pt]{amsart}%
\usepackage[T1]{fontenc}
\usepackage[latin9]{inputenc}
\usepackage{amsmath}
\usepackage{amssymb}
\usepackage{amsfonts}
\usepackage{babel}
\usepackage{fancyhdr}
\usepackage{hyperref}
\usepackage{color}
\usepackage[toc]{appendix}
\usepackage{graphicx}
\usepackage[singlelinecheck=true]{caption}%
\providecommand{\U}[1]{\protect\rule{.1in}{.1in}}
\hypersetup{colorlinks=true, linkcolor=blue, filecolor=blue, pagecolor=blue, urlcolor=blue}
\makeatletter
\newtheorem{theorem}{Theorem}

\newtheorem{lemma}[theorem]{Lemma}

\newtheorem{definition}[theorem]{Definition}
\newtheorem{condition}[theorem]{Condition}
\newtheorem{preremark}[theorem]{Remark}
\newenvironment{remark}    {\begin{preremark}\rm}{\end{preremark}}
\newtheorem{preexample}{Example}
\newenvironment{example}    {\begin{preexample}\rm}{\end{preexample}}
\numberwithin{equation}{section}
\numberwithin{preexample}{section}
\numberwithin{theorem}{section}

\makeatother
\begin{document}

\title{Infinite Swapping using IID Samples}

\author{Paul Dupuis, Guo-Jhen Wu, Michael Snarski}
\thanks{Research supported in part by DOE
DE-SC-0010539 and DARPA EQUiPS W911NF-15-2-0122, NSF DMS-1317199 (PD), and NSERC
PGSD3-471576-2015(MS)}

\maketitle

\begin{abstract}
We propose a new method for estimating rare event probabilities when
independent samples are available. It is assumed that the underlying
probability measures satisfy a large deviations principle with a scaling
parameter $\varepsilon$ that we call temperature. We show how by combining
samples at different temperatures, one can construct an estimator with
greatly reduced variance. Although as presented here the method is not as
broadly applicable as other rare event simulation methods, such as splitting
or importance sampling, it does not require any problem-dependent
constructions.
\end{abstract}

\section{Introduction}

The purpose of this note is to introduce a new method for estimating
probabilities of rare events and expected values that are largely determined
by rare events. In the form developed here, the method makes significant
assumptions on the probability measures involved, and for this reason is
less broadly applicable than other methods such as importance sampling and
splitting. However, when applicable it is remarkably easy to use, and gives
large variance reduction while requiring very little in the way of
preliminary constructions, such as importance functions.

The method assumes one can sample directly from the target distribution and
certain related distributions. In particular, the generation of samples is
not based on the simulation of a Markov chain. However, it is directly
motivated by a scheme known as infinite swapping (INS) that is applied to
systems sampled using MCMC and for which one does not sample directly from
the target distribution. Although the focus of this paper will be to
introduce and optimize the scheme in a large deviation limit when direct
sampling is possible, the lessons learned here can be applied to the more
generally applicable INS scheme, and this will be a topic for future work.

An outline of the paper is as follows. In the next section we state
assumptions, specify the estimation problem of interest, and introduce the
estimator. Section \ref{sec:performance} analyzes the performance of the
scheme in the large deviation limit by characterizing the decay rate of the
second moment of the estimator as a function of certain design parameters,
and then shows how these parameters can be selected to optimize the decay
rate. Section \ref{sec:extensions} presents generalizations of the
framework, and computational examples are given in Section \ref{sec:examples}%
. The method has certain features in common with schemes often used for
these problems (importance sampling and multi-level splitting), and a
discussion on these points is given in Section \ref{sec:relations}. We end
with a section of concluding remarks and an appendix that gives proofs of
various supporting results.

\section{Problem Formulation and Definition of the Estimator}

\label{sec:formulation}

\subsection{Problems of interest}

Consider a collection of probability measures $\{\mu^{\varepsilon}\}_{%
\varepsilon \in(0,\infty)}$ on a Polish space $\mathcal{X}$ (i.e., a
complete seperable metric space). A rate function is a mapping $I:\mathcal{X}
\rightarrow [0,\infty]$ such that $\{x:I(x) \leq M\}$ is compact for each $M
\in [0,\infty)$. We assume $\{\mu^{\varepsilon}\}$ satisfies the following
properties:

\begin{enumerate}
\item $\{\mu^{\varepsilon}\}$ satisfies a large deviation principle (LDP) 
\cite{dupell4} with rate function $I$ that is continuous on its domain of
finiteness and scaling $\varepsilon$,

\item the measures $\mu^{\varepsilon/\alpha_{1}}$ and $\mu^{\varepsilon
/\alpha_{2}}$ are mutually absolutely continuous for $\alpha_{1},\alpha_{2}%
\in(0,1]$, and we can explicitly compute quantities such as 
\begin{equation*}
\frac{\mu^{\varepsilon/\alpha_{1}}(dx_{1})\mu^{\varepsilon/%
\alpha_{2}}(dx_{2})}{\mu^{\varepsilon/\alpha_{1}}(dx_{1})\mu^{\varepsilon/%
\alpha_{2}}(dx_{2})+\mu^{\varepsilon/\alpha_{2}}(dx_{1})\mu^{\varepsilon/%
\alpha_{1}}(dx_{2})}, 
\end{equation*}
i.e., likelihood ratios for the corresponding product measure under
permutations.
\end{enumerate}

\begin{example}
\label{exa:1}Consider the case $\mathcal{X}=\mathbb{R}^{d}$, and 
\begin{equation*}
\mu^{\varepsilon}\left( dx\right) =\frac{1}{Z^{\varepsilon}}e^{-\frac {1}{%
\varepsilon}I(x)}dx,\quad Z^{\varepsilon}\doteq\int_{\mathbb{R}^{d}}e^{-%
\frac{1}{\varepsilon}I(x)}dx, 
\end{equation*}
where $I:$\ $\mathbb{R}^{d}\rightarrow\lbrack0,\infty)$\ is a continuous
rate function and 
\begin{equation*}
\liminf_{M\rightarrow\infty}\inf_{x:\left\Vert x\right\Vert =M}\frac{I\left(
x\right) }{M}>0 
\end{equation*}
(the latter condition in particular guarantees that $Z^{\varepsilon}$\ is
finite). Then $\{\mu^{\varepsilon}\}$\ automatically satisfies the LDP with
rate function $I$\ (see Theorem \ref{thm:ldp} in the Appendix). A basic
example is $I(x)=\frac{1}{2}\left\langle x-b,B^{-1}(x-b)\right\rangle $\
with $B\geq0$\ symmetric, where $\left\langle \cdot,\cdot\right\rangle $\ is
the standard inner product on $\mathbb{R}^{d}$ , in which case the
associated $\mu^{\varepsilon}$\ is Gaussian with mean $b$\ and covariance
matrix $\varepsilon B$.
\end{example}

Additional examples are given in Section \ref{sec:examples}. We are
interested in estimating \textit{rare event} probabilities such as 
\begin{equation}
\mu^{\varepsilon}(A)   \label{eqn:setAprob}
\end{equation}
for a set $A\subset\mathcal{X}$ that does not contain any of the minimizers
of $I$, and \textit{risk-sensitive} functionals \cite{whi} of the form%
\begin{equation}
\int_{\mathcal{X}}e^{-\frac{1}{\varepsilon}F^{\varepsilon}(x)}\mu
^{\varepsilon}(dx)   \label{eqn:riskfunction}
\end{equation}
for various classes of functions $F^{\varepsilon}$.

\subsection{The INS estimator}

To introduce the estimator we focus first on the case of risk-sensitive
functionals as in (\ref{eqn:riskfunction}). We also simplify by assuming
that $\{\mu^{\varepsilon}\}$ satisfies conditions similar to Example \ref%
{exa:1}, and that $F^{\varepsilon}=F$ for all $\varepsilon$. Later in
Section \ref{sec:extensions}, we outline extensions where the measures do
not take this exact form and $F^{\varepsilon}$ depends on $\varepsilon$.

\begin{condition}
\label{con:1}$\{\mu^{\varepsilon}\}$ takes the form $\mu^{\varepsilon }(dx)=%
\frac{1}{Z^{\varepsilon}}e^{-\frac{1}{\varepsilon}I(x)}\gamma(dx)$ for a
reference measure $\gamma$ on $\mathcal{X}$ and satisfies the LDP with rate $%
I$.
\end{condition}

The definition of the INS estimator requires certain parameters, which we
now introduce. Given a fixed $K\in\mathbb{%
\mathbb{N}
}$, consider numbers $1=\alpha_{1}\geq\alpha_{2}\geq\cdots\geq\alpha_{K}\geq
0$. In a physical system $\varepsilon$ would be a parameter that plays the
role of temperature, and the parameters $\varepsilon/\alpha_{i},i>1$ define
alternative temperatures (larger than the starting temperature since $%
\alpha_{i}\leq1$). The terminology \textit{infinite swapping} is inherited
from an analogous estimator constructed in a dynamical setting, in which
swapping of states allows parallel replicas of the system at different
temperatures to be linked. In the present setting, this linkage will be
accomplished through appropriate weights. Letting $\boldsymbol{\alpha}%
\doteq(\alpha_{1},\alpha_{2},\ldots,\alpha_{K})$ and $\mathbf{x}\doteq
(x_{1},x_{2},\ldots,x_{K})$ with each $x_{j}\in\mathcal{X}$, we define the
weights 
\begin{equation}
\rho^{\varepsilon}\left( \mathbf{x};\boldsymbol{\alpha}\right) \doteq \frac{%
e^{-\frac{1}{\varepsilon}\sum_{j=1}^{K}\alpha_{j}I\left( x_{j}\right) }}{%
\sum_{\tau\in\Sigma_{K}}e^{-\frac{1}{\varepsilon}\sum_{j=1}^{K}\alpha
_{j}I\left( x_{\tau\left( j\right) }\right) }}.   \label{eqn:defofrho}
\end{equation}
Let $\{X_{j}^{\varepsilon}\}_{j\in\{1,\ldots,K\}}$ be independent with $%
X_{j}^{\varepsilon}\sim\mu^{\varepsilon/\alpha_{j}}$, where $\sim$ means
that $X_{j}^{\varepsilon}$ has the distribution $\mu^{\varepsilon/\alpha_{j}}
$. Let $\Sigma_{K}$ denote the set of permutations on $\{1,\ldots,K\}$, and
for $\sigma\in\Sigma_{K}$ let $\mathbf{X}_{\sigma}^{\varepsilon}$ denote $%
(X_{\sigma\left( 1\right) }^{\varepsilon},\ldots,X_{\sigma\left( K\right)
}^{\varepsilon})$. Then a single sample of the $K$-temperature INS estimator
based on $\boldsymbol{\alpha}$\ for estimating $\int_{\mathcal{X}}e^{-\frac {%
1}{\varepsilon}F(x)}\mu^{\varepsilon}(dx)$ is 
\begin{equation*}
\hat{\eta}_{\mbox{\tiny{INS}}}^{\varepsilon}\left( \mathbf{X}^{\varepsilon };%
\boldsymbol{\alpha}\right)
=\sum\nolimits_{\sigma\in\Sigma_{K}}\rho^{\varepsilon }\left( \mathbf{X}%
_{\sigma}^{\varepsilon};\boldsymbol{\alpha}\right) e^{-\frac{1}{\varepsilon}%
F\left( X_{\sigma\left( 1\right) }^{\varepsilon }\right) }. 
\end{equation*}
Although $\hat{\eta}_{\mbox{\tiny{INS}}}^{\varepsilon}$ depends on the
number $K$ of temperatures, to simplify notation this is not made explicit.
The analogous estimator when applied to the problem of approximating
probabilities [i.e., $F(x)=\infty 1_{A^{c}}(x)$] will be denoted by $\hat{%
\theta }_{\mbox{\tiny{INS}}}^{\varepsilon}$. Since we discuss various
unbiased estimators, a subscript (e.g., INS) is sometimes used to
distinguish different estimators. If the subscript is missing, then it is
always the INS estimator that is meant.

\begin{remark}
We will always assume $\boldsymbol{\alpha}$\ satisfies the ordering $%
1=\alpha_{1}\geq\alpha_{2}\geq\cdots\geq\alpha_{K}\geq0$.
\end{remark}

Before proceeding we establish basic properties of the estimator.

\begin{lemma}
Assume Condition \ref{con:1}. Then $\hat{\eta}^{\varepsilon}\left(
\mathbf{X}^{\varepsilon};\boldsymbol{\alpha}\right)  $ is an unbiased
estimator of $\int_{\mathcal{X}}e^{-\frac{1}{\varepsilon}F(x)}\mu
^{\varepsilon}(dx)$.
\end{lemma}

\begin{proof}
It follows from the definitions that
\begin{align*}
  E\left(  \hat{\eta}^{\varepsilon}\left(  \mathbf{X}^{\varepsilon
};\boldsymbol{\alpha}\right)  \right)  
&  =\frac{1}{Z^{\varepsilon;\boldsymbol{\alpha}}}\int_{\mathcal{X}^{K}%
}\sum\nolimits_{\sigma\in\Sigma_{K}}\rho^{\varepsilon}\left(  \mathbf{x}%
_{\sigma};\boldsymbol{\alpha}\right)  e^{-\frac{1}{\varepsilon}F\left(
x_{\sigma\left(  1\right)  }\right)  }e^{-\frac{1}{\varepsilon}\sum_{j=1}%
^{K}\alpha_{j}I\left(  x_{j}\right)  }\gamma^{K}\left(  d\mathbf{x}\right)  \\
&  =\sum_{\sigma\in\Sigma_{K}}\frac{1}{Z^{\varepsilon;\boldsymbol{\alpha
}}}\int_{\mathcal{X}^{K}}\frac{e^{-\frac{1}{\varepsilon}\sum_{j=1}^{K}%
\alpha_{j}I\left(  x_{_{\sigma\left(  j\right)  }}\right)  }e^{-\frac
{1}{\varepsilon}F\left(  x_{\sigma\left(  1\right)  }\right)  }}{\sum\nolimits_{\tau
\in\Sigma_{K}}e^{-\frac{1}{\varepsilon}\sum_{j=1}^{K}\alpha_{j}I\left(
x_{\tau\left(  j\right)  }\right)  }}e^{-\frac{1}{\varepsilon}\sum_{j=1}%
^{K}\alpha_{j}I\left(  x_{j}\right)  }\gamma^{K}\left(  d\mathbf{x}\right)  ,
\end{align*}
where%
\[
\gamma^{K}\left(  d\mathbf{x}\right)  \doteq\gamma\left(  dx_{1}\right)
\cdots\gamma\left(  dx_{K}\right)  \text{ and }Z^{\varepsilon
;\boldsymbol{\alpha}}\doteq Z^{\varepsilon}Z^{\varepsilon/\alpha_{1}}\cdots
Z^{\varepsilon/\alpha_{K}.}%
\]
For each $\sigma$ the change of variables from $(x_{\sigma\left(  1\right)
},\ldots,x_{\sigma\left(  K\right)  })$ to $\left(  y_{1},\ldots,y_{K}\right)
$ gives
\begin{align*}
&  \int_{\mathcal{X}^{K}}\frac{e^{-\frac{1}{\varepsilon}\sum_{j=1}^{K}%
\alpha_{j}I\left(  x_{_{\sigma\left(  j\right)  }}\right)  }e^{-\frac
{1}{\varepsilon}F\left(  x_{\sigma\left(  1\right)  }\right)  }}{\sum_{\tau
\in\Sigma_{K}}e^{-\frac{1}{\varepsilon}\sum_{j=1}^{K}\alpha_{j}I\left(
x_{\tau\left(  j\right)  }\right)  }}e^{-\frac{1}{\varepsilon}\sum_{j=1}%
^{K}\alpha_{j}I\left(  x_{j}\right)  }\gamma^{K}\left(  d\mathbf{x}\right)  \\
&  \quad=\int_{\mathcal{X}^{K}}\frac{e^{-\frac{1}{\varepsilon}\sum_{j=1}%
^{K}\alpha_{j}I\left(  y_{j}\right)  }e^{-\frac{1}{\varepsilon}F\left(
y_{1}\right)  }}{\sum_{\tau\in\Sigma_{K}}e^{-\frac{1}{\varepsilon}\sum
_{j=1}^{K}\alpha_{j}I\left(  y_{\tau^{-1}\left(  j\right)  }\right)  }%
}e^{-\frac{1}{\varepsilon}\sum_{j=1}^{K}\alpha_{j}I\left(  y_{\sigma
^{-1}\left(  j\right)  }\right)  }\gamma^{K}\left(  d\mathbf{y}\right)  .
\end{align*}
Thus%
\begin{align*}
  E\left(  \hat{\eta}^{\varepsilon}\left(  \mathbf{X}^{\varepsilon
};\boldsymbol{\alpha}\right)  \right)  
&  =\sum_{\sigma\in\Sigma_{K}}\int_{\mathcal{X}^{K}}\frac{e^{-\frac
{1}{\varepsilon}\sum_{j=1}^{K}\alpha_{j}I\left(  y_{j}\right)  }e^{-\frac
{1}{\varepsilon}F\left(  y_{1}\right)  }}{\sum_{\tau\in\Sigma_{K}}e^{-\frac
{1}{\varepsilon}\sum_{j=1}^{K}\alpha_{j}I\left(  y_{\tau^{-1}\left(  j\right)
}\right)  }}\frac{e^{-\frac{1}{\varepsilon}\sum_{j=1}^{K}\alpha_{j}I\left(
y_{\sigma^{-1}\left(  j\right)  }\right)  }}{Z^{\varepsilon;\boldsymbol{\alpha
}}}\gamma^{K}\left(  d\mathbf{y}\right)  \\
&  =\frac{1}{Z^{\varepsilon;\boldsymbol{\alpha}}}\int_{\mathcal{X}^{K}%
}e^{-\frac{1}{\varepsilon}\sum_{j=1}^{K}\alpha_{j}I\left(  y_{j}\right)
}e^{-\frac{1}{\varepsilon}F\left(  y_{1}\right)  }\gamma^{K}\left(
d\mathbf{y}\right)  \\
&  =\int_{\mathcal{X}}e^{-\frac{1}{\varepsilon}F^{\varepsilon}(x)}%
\frac{1}{Z^{\varepsilon}}e^{-\frac{1}{\varepsilon}I(x)}\gamma\left(
dx\right)  .
\end{align*}

\end{proof}

\medskip A useful observation is that if $\mathbf{Y}^{\varepsilon}$ is
sampled from the symmetrized version of the distribution of $\mathbf{X}%
^{\varepsilon}$, i.e., if \textsl{\ }%
\begin{equation*}
\mathbf{Y}^{\varepsilon}\sim\frac{\sum_{\sigma\in\Sigma_{K}}e^{-\frac {1}{%
\varepsilon}\sum_{j=1}^{K}\alpha_{j}I\left( y_{\sigma\left( j\right)
}\right) }}{K!Z^{\varepsilon;\boldsymbol{\alpha}}}\gamma^{K}\left( d\mathbf{y%
}\right) , 
\end{equation*}
then with $\overset{d}{=}$ meaning that two random variables have the same
distribution, 
\begin{equation*}
\hat{\eta}^{\varepsilon}\left( \mathbf{X}^{\varepsilon};\boldsymbol{\alpha }%
\right) \overset{d}{=}\hat{\eta}^{\varepsilon}\left( \mathbf{Y}%
^{\varepsilon};\boldsymbol{\alpha}\right) . 
\end{equation*}
Indeed, for any $t\in\mathbb{R}$%
\begin{align*}
E\left( e^{it\hat{\eta}^{\varepsilon}\left( \mathbf{Y}^{\varepsilon };%
\boldsymbol{\alpha}\right) }\right) & =\int_{\mathcal{X}^{K}}e^{it\hat{\eta}%
^{\varepsilon}\left( y_{1,\ldots ,}y_{K};\boldsymbol{\alpha}\right) }\frac{%
\sum_{\sigma\in\Sigma_{K}}e^{-\frac{1}{\varepsilon}\sum_{j=1}^{K}\alpha_{j}I%
\left( y_{\sigma\left( j\right) }\right) }}{K!Z^{\varepsilon;\boldsymbol{%
\alpha}}}\gamma^{K}\left( d\mathbf{y}\right) \\
& =\frac{1}{K!Z^{\varepsilon;\boldsymbol{\alpha}}}\sum\nolimits_{\sigma
\in\Sigma_{K}}\int_{\mathcal{X}^{K}}e^{it\hat{\eta}^{\varepsilon}\left(
y_{1,\ldots,}y_{K};\boldsymbol{\alpha}\right) }e^{-\frac{1}{\varepsilon}%
\sum_{j=1}^{K}\alpha_{j}I\left( y_{j}\right) }\gamma^{K}\left( d\mathbf{y}%
\right) \\
& =\frac{1}{Z^{\varepsilon;\boldsymbol{\alpha}}}\int_{\mathcal{X}^{K}}e^{it%
\hat{\eta}^{\varepsilon}\left( y_{1,\ldots,}y_{K};\boldsymbol{\alpha }%
\right) }e^{-\frac{1}{\varepsilon}\sum_{j=1}^{K}\alpha_{j}I\left(
y_{j}\right) }\gamma^{K}\left( d\mathbf{y}\right) \\
& =E\left( e^{it\hat{\eta}^{\varepsilon}\left( \mathbf{X}^{\varepsilon };%
\boldsymbol{\alpha}\right) }\right) ,
\end{align*}
where the fact that $\hat{\eta}^{\varepsilon}\left( \mathbf{x};\boldsymbol{%
\alpha}\right) $ is invariant under permutations of $\mathbf{x}$ is used for
the second equality.

\begin{remark}
For implementation purposes one prefers $\mathbf{X}^{\varepsilon}$\ over $%
\mathbf{Y}^{\varepsilon}$, since it is easier to generate samples of $%
\mathbf{X}^{\varepsilon}$\ than $\mathbf{Y}^{\varepsilon}$. However, we
introduce $\mathbf{Y}^{\varepsilon}$\ because it is directly analogous to
the quantity actually sampled in the dynamical setting which inspired the
present static form of INS, and also due to the fact that it will simplify
proofs later on. Specifically, we will use that $\left\{ \mathbf{Y}%
^{\varepsilon }\right\} _{\varepsilon\in(0,1)}$\ satisfies the LDP with rate
function $\min_{\tau\in\Sigma_{K}}\{\sum_{j=1}^{K}\alpha_{j}I\left(
x_{\tau\left( j\right) }\right) \}$\ (see Lemma \ref{Lem4}). This rate
function gives a simpler expression for the decay rate of $E(\hat{\eta}%
^{\varepsilon}\left( \mathbf{Y}^{\varepsilon};\boldsymbol{\alpha}\right)
)^{2}$, which in turn allows a simpler proof of our main result concerning
the decay rate of $E(\hat{\eta}^{\varepsilon}\left( \mathbf{X}^{\varepsilon};%
\boldsymbol{\alpha }\right) )^{2}$.
\end{remark}

\section{Analysis of Performance}

\label{sec:performance}

In this section we consider the performance of the INS estimator. Since we
assume that one can generate independent samples $\left\{ X^{\varepsilon
}\left( i\right) \right\} _{i\in\mathbb{N}}$ from the target measure $%
\mu^{\varepsilon}$, a possible estimator for (\ref{eqn:setAprob}) is the
straightforward Monte Carlo scheme 
\begin{equation*}
\hat{\theta}_{\mbox{\tiny{MC}}}^{\varepsilon}(N)=\frac{1}{N}\sum
\nolimits_{i=1}^{N}1_{\left\{ X^{\varepsilon}\left( i\right) \in A\right\}
}. 
\end{equation*}
The estimator $\hat{\theta}_{\mbox{\tiny{MC}}}^{\varepsilon}(N)$ is unbiased
and by the law of large numbers converges to $\mu^{\varepsilon}(A)$ almost
surely as $N\rightarrow\infty$. However, as is well known its relative error
can make it impractical, since 
\begin{equation*}
\frac{\sqrt{\text{Var}(\hat{\theta}_{\mbox{\tiny{MC}}}^{\varepsilon}(N))}}{E[%
\hat{\theta}_{\mbox{\tiny{MC}}}^{\varepsilon}(N)]}=\frac{\sqrt {%
\mu^{\varepsilon}(A)(1-\mu^{\varepsilon}(A))}}{\sqrt{N}\mu^{\varepsilon}(A)}%
\approx\frac{1}{\sqrt{N\mu^{\varepsilon}(A)}}. 
\end{equation*}
For small $\varepsilon$, under mild conditions on $\partial A$ we have the
Laplace type approximation $\mu^{\varepsilon}(A)\approx e^{-\frac {1}{%
\varepsilon}\inf_{x\in A}I(x)}$. Thus to maintain a bounded relative error
one needs exponentially many samples $N$.

As an alternative one can consider any estimator of the form 
\begin{equation}
\frac{1}{N}\sum\nolimits_{i=1}^{N}R_{i}^{\varepsilon},   \label{eqn:genest}
\end{equation}
with the $\left\{ R_{i}^{\varepsilon}\right\} $ independent and identically
distributed (iid) and $ER_{1}^{\varepsilon}=\mu^{\varepsilon}(A)$ so the
estimator is unbiased. Since the variance of (\ref{eqn:genest}) is
proportional to that of $R_{i}^{\varepsilon}$, it is enough to consider the
single sample estimator $R^{\varepsilon}$ when comparing performance. Also,
since all estimators under consideration are unbiased, comparing variances
reduces to comparing second moments. Similar considerations apply when
considering risk-sensitive expected values, for which probabilities can be
considered a special case by taking $F(x)=\infty \cdot 1_{A^{c}}(x)$, where $%
1_{A^{c}}$ is the indicator of the complement of $A$.

Suppose that $\hat{\eta}^{\varepsilon}$ is a single sample estimator of $%
\int_{\mathcal{X}}e^{-\frac{1}{\varepsilon}F(x)}\mu^{\varepsilon}(dx)$, with 
$F$ bounded below and continuous. By standard Laplace asymptotics \cite[%
Theorem 1.3.4]{dupell4} 
\begin{equation*}
I(F)\doteq\inf_{x\in\mathcal{X}}\left[ F(x)+I(x)\right]=\lim_{\varepsilon%
\rightarrow0}-\varepsilon\log\int_{\mathcal{X}}e^{-\frac {1}{\varepsilon}%
F(x)}\mu^{\varepsilon}(dx), 
\end{equation*}
and the expected value is largely determined by tail properties of $%
\mu^{\varepsilon}$ if $I(x^{\ast})>0$ for all $x^{\ast}$ that minimize in $%
I(F)$. When this holds, one way of assessing the performance of $\hat{\eta }%
^{\varepsilon}$ is to obtain bounds on the limit of $-\varepsilon\log E[(%
\hat{\eta}^{\varepsilon})^{2}]$ (i.e., the decay rate), or at least lower
bounds on 
\begin{equation*}
\liminf_{\varepsilon\rightarrow0}-\varepsilon\log E[(\hat{\eta}^{\varepsilon
})^{2}], 
\end{equation*}
with larger lower bounds indicating better performance. As is well known
there is a limit to how well an estimator can perform. By Jensen's inequality%
\begin{equation*}
\liminf_{\varepsilon\rightarrow0}-\varepsilon\log E[(\hat{\eta}^{\varepsilon
})^{2}]\leqslant\liminf_{\varepsilon\rightarrow0}-2\varepsilon\log \int_{%
\mathcal{X}}e^{-\frac{1}{\varepsilon}F(x)}\mu^{\varepsilon}(dx)=2I(F). 
\end{equation*}
An unbiased estimator $\hat{\eta}^{\varepsilon}$ is said to be \textit{%
asymptotically optimal} if the upper bound is achieved, i.e., 
\begin{equation*}
\lim_{\varepsilon\rightarrow0}-\varepsilon\log E[(\hat{\eta}^{\varepsilon
})^{2}]=2I(F). 
\end{equation*}

As a heuristic motivation for the INS estimator, consider the problem of
approximating%
\begin{equation*}
P\left( X^{\varepsilon}\in A\right) \approx e^{-\frac{1}{\varepsilon }%
I\left( a\right) }, 
\end{equation*}
where $0$\ is the unique minimizer of $I$\ and $0\notin A=[a,\infty)$, $a$\
is the minimizer of $I$\ over $[a,\infty)$, and where $\approx$\ indicates
that the left and right hand sides have the same rate of decay as $%
\varepsilon \rightarrow0$. Consider the case of 2 temperatures. With $%
\alpha_{1}=1$\ and $\alpha_{2}=\alpha$ 
\begin{equation*}
\hat{\theta}_{\text{INS}}^{\varepsilon}=\rho^{\varepsilon}\left(
X_{1}^{\varepsilon},X_{2}^{\varepsilon}\right) 1_{A}\left(
X_{1}^{\varepsilon}\right) +\rho^{\varepsilon}\left(
X_{2}^{\varepsilon},X_{1}^{\varepsilon}\right) 1_{A}\left(
X_{2}^{\varepsilon}\right) , 
\end{equation*}
where\emph{\ }%
\begin{equation*}
X_{1}^{\varepsilon}\sim\frac{1}{Z^{\varepsilon}}e^{-\frac{1}{\varepsilon }%
I\left( x\right) }dx\text{ and }X_{2}^{\varepsilon}\sim\frac{1}{%
Z^{\varepsilon/\alpha}}e^{-\frac{\alpha}{\varepsilon}I\left( x\right) }dx 
\end{equation*}
and 
\begin{equation*}
\rho^{\varepsilon}\left( x,y\right) =\frac{e^{-\frac{1}{\varepsilon}\left(
I\left( x\right) +\alpha I\left( y\right) \right) }}{e^{-\frac {1}{%
\varepsilon}\left( I\left( x\right) +\alpha I\left( y\right) \right) }+e^{-%
\frac{1}{\varepsilon}\left( I\left( y\right) +\alpha I\left( x\right)
\right) }}. 
\end{equation*}
Thus for small $\varepsilon>0$, there are three types of outcomes:

\begin{itemize}
\item $\hat{\theta}_{\text{INS}}^{\varepsilon}=1$\ when $\left(
X_{1}^{\varepsilon},X_{2}^{\varepsilon}\right) \in A\times A,$\ which occurs
with approximate probability $P\left( \left(
X_{1}^{\varepsilon},X_{2}^{\varepsilon}\right) \in A\times A\right) \approx
e^{-\frac {1}{\varepsilon}I\left( a\right) }\cdot e^{-\frac{1}{\varepsilon}%
\alpha I\left( a\right) }=e^{-\frac{1}{\varepsilon}\left( 1+\alpha\right)
I\left( a\right) }$.

\item $\hat{\theta}_{\text{INS}}^{\varepsilon}=0$\ when $\left(
X_{1}^{\varepsilon},X_{2}^{\varepsilon}\right) \in A^{c}\times A^{c},$\ with
approximate probability $P\left( \left(
X_{1}^{\varepsilon},X_{2}^{\varepsilon}\right) \in A^{c}\times A^{c}\right)
\approx1$.

\item $\hat{\theta}_{\text{INS}}^{\varepsilon}\approx\rho^{\varepsilon}%
\left( a,0\right) $\ when $\left(
X_{1}^{\varepsilon},X_{2}^{\varepsilon}\right) \in A\times A^{c}$\ or $%
A^{c}\times A$, with approximate probability\emph{\ }%
\begin{equation}
P\left( \left( X_{1}^{\varepsilon},X_{2}^{\varepsilon}\right) \in A\times
A^{c}\right) +P\left( \left( X_{1}^{\varepsilon},X_{2}^{\varepsilon}\right)
\in A^{c}\times A\right) \approx e^{-\frac{1}{\varepsilon}I\left( a\right)
}+e^{-\frac {1}{\varepsilon}\alpha I\left( a\right) }\approx e^{-\frac{1}{%
\varepsilon }\alpha I\left( a\right) }.  \label{eqn:heuristic1}
\end{equation}
\end{itemize}

Using the definition of $\rho^{\varepsilon}$\ gives 
\begin{equation*}
\rho^{\varepsilon}\left( a,0\right) =\frac{e^{-\frac{1}{\varepsilon}I\left(
a\right) }}{e^{-\frac{1}{\varepsilon}I\left( a\right) }+e^{-\frac {1}{%
\varepsilon}\alpha I\left( a\right) }}\approx e^{-\frac{1}{\varepsilon }%
\left( 1-\alpha\right) I\left( a\right) }, 
\end{equation*}
and therefore%
\begin{align*}
E(\hat{\theta}_{\text{INS}}^{\varepsilon})^{2} & \approx1^{2}\cdot e^{-\frac{%
1}{\varepsilon}\left( 1+\alpha\right) I\left( a\right) }+\left(
\rho^{\varepsilon}\left( a,0\right) \right) ^{2}\cdot e^{-\frac {1}{%
\varepsilon}\alpha I\left( a\right) } \\
& \approx e^{-\frac{1}{\varepsilon}\left( 1+\alpha\right) I\left( a\right)
}+e^{-\frac{1}{\varepsilon}\left( 2-\alpha\right) I\left( a\right) } \\
& \approx e^{-\frac{1}{\varepsilon}\left[ \left( 1+\alpha\right)
\wedge\left( 2-\alpha\right) \right] I\left( a\right) }.
\end{align*}
Thus 
\begin{equation*}
\lim_{\varepsilon\rightarrow0}-\varepsilon\log E(\hat{\theta}_{\text{INS}%
}^{\varepsilon})^{2}=\left( \left[ \left( 1+\alpha\right) \wedge\left(
2-\alpha\right) \right] \right) I\left( a\right) , 
\end{equation*}
with the maximal rate of decay of $3I(a)/2$\ at $\alpha=1/2$, which can be
compared with the rate of decay $I(a)$\ that would be found for standard
Monte Carlo.

\subsection{Approximating a risk sensitive functional}

In this section we evaluate the decay rate for the INS estimator and the
optimal $\boldsymbol{\alpha}$ for a fixed number of temperatures. The first
step is to define the function $V_{F}(\boldsymbol{\alpha})$ that
characterizes the decay rate for a fixed parameter $\boldsymbol{\alpha}$ and
solve for the optimum over $\boldsymbol{\alpha}$. The proof of the following
lemma will be given after stating and proving the asymptotic bounds.

\begin{lemma}
\label{Lem2}Let $F:\mathcal{X}\rightarrow\mathbb{R}$ be continuous and bounded
below. Let $I:\mathcal{X}\rightarrow\lbrack0,\infty]$ be continuous on its
domain of finiteness. For fixed $K\in\mathbb{%
\mathbb{N}
}$ and $\boldsymbol{\alpha}$\textbf{\ }with $1=\alpha_{1}\geq\alpha_{2}%
\geq\cdots\geq\alpha_{K}\geq0$, let $D$ $\doteq\left\{  \left(  I\left(
x\right)  ,F\left(  x\right)  \right)  :x\in\mathcal{X}\right\}  $ and define
\[
V_{F}(\boldsymbol{\alpha})\doteq\inf_{\substack{\left(  I_{1},F\right)  \in
D\\\left\{  \boldsymbol{I}:I_{j}\in\lbrack0,I_{1}]\text{ for }%
j\geq2\right\}  }}\left[  2\sum\nolimits_{j=1}^{K}\alpha_{j}I_{j}+2F-\min_{\sigma
\in\Sigma_{K}}\left\{  \sum\nolimits_{j=1}^{K}\alpha_{j}I_{\sigma\left(  j\right)
}\right\}  \right]  . 
\]
Then%
\[
\inf_{\mathbf{x\in}\mathcal{X}^{K}}\left[  2\sum\nolimits_{j=1}^{K}\alpha_{j}I\left(
x_{_{j}}\right)  +2F\left(  x_{1}\right)  -\min_{\sigma\in\Sigma_{K}}\left\{
\sum\nolimits_{j=1}^{K}\alpha_{j}I\left(  x_{\sigma\left(  j\right)  }\right)
\right\}  \right]  =V_{F}(\boldsymbol{\alpha}).
\]
Moreover%
\[
\sup_{\boldsymbol{\alpha}}V_{F}\left(  \boldsymbol{\alpha}\right)
=\inf_{x\mathbf{\in}\mathcal{X}}\left\{  \left(  2-\left(  1/2\right)
^{K-1}\right)  I\left(  x\right)  +2F\left(  x\right)  \right\}  ,
\]
and $\boldsymbol{\alpha}^{\ast}=(1,1/2,1/4,\ldots,\left(  1/2\right)  ^{K-1})$
achieves the supremum.
\end{lemma}

When combined with the following result, Lemma \ref{Lem2} shows that $\hat{%
\eta}^{\varepsilon}\left( \mathbf{X}^{\varepsilon};\boldsymbol{\alpha }%
\right) $ reaches optimal decay rate when $\boldsymbol{\alpha}=\boldsymbol{%
\alpha}^{\ast}.$

\begin{theorem}
\label{Thm1}Assume Condition \ref{con:1}, and let $F:\mathcal{X}\rightarrow%
\mathbb{R}$ be continuous and bounded below. Let $I:\mathcal{X}%
\rightarrow\lbrack0,\infty]$ be continuous on its domain of finiteness. Then
for any $K\in%
\mathbb{N}
$ and $\boldsymbol{\alpha}$ with $1=\alpha_{1}\geq\alpha_{2}\geq\cdots
\geq\alpha_{K}\geq0$, 
\begin{equation*}
\lim_{\varepsilon\rightarrow0}-\varepsilon\log E\left( \hat{\eta }%
^{\varepsilon}\left( \mathbf{X}^{\varepsilon};\boldsymbol{\alpha}\right)
\right) ^{2}=V_{F}\left( \boldsymbol{\alpha}\right) . 
\end{equation*}
\end{theorem}

\begin{remark}
By formally applying this result with $F(x)=-\infty1_{A^{c}}(x)$\ we obtain
the optimal decay rates for estimating the probability of the set $A$\ as in
(\ref{eqn:setAprob}), a result that will be stated precisely later on. In
this case the decay rate with the optimal $\boldsymbol{\alpha}$\ is $%
(2-\left( 1/2\right) ^{K-1})\inf_{x\in A}I\left( x\right) $, while the best
possible is $2\inf_{x\in A}I\left( x\right) $. With a continuous,
non-negative $F$ one does even better in some sense, since 
\begin{equation*}
\inf_{x\in\mathcal{X}}\left\{ \left( 2-\left( 1/2\right) ^{K-1}\right)
I\left( x\right) +2F\left( x\right) \right\} \geq\left( 2-\left( 1/2\right)
^{K-1}\right) I(F). 
\end{equation*}
\end{remark}

\begin{remark}
To evaluate a $K$ temperature INS estimator $\hat{\eta}^{\varepsilon}\left( 
\mathbf{X}^{\varepsilon};\boldsymbol{\alpha}\right) $ we need to sum over $K!
$ terms, which in general will be computationally expensive. However, for
nonnegative $F,$ if $K=5$ and $\boldsymbol{\alpha}=\boldsymbol{\alpha}^{\ast}
$ then the decay rate of $E(\hat{\eta}^{\varepsilon}\left( \mathbf{X}%
^{\varepsilon};\boldsymbol{\alpha }\right) )^{2}$ is guaranteed to be at
least $1.9375\cdot I(F)$, while $K=6$ pushes the rate up to nearly $%
1.97\cdot I(F)$. These are both close to the optimum of $2\cdot I(F)$, and
require summing over $120$ and $720$ terms, respectively. In general, one
should balance $\varepsilon$ against $K$ to decide on an appropriate value
of $K$.
\end{remark}

The proof of Theorem \ref{Thm1} is based on establishing that $V_{F}(%
\boldsymbol{\alpha})$ is an upper bound for%
\begin{equation*}
\limsup\limits_{\varepsilon\rightarrow0}-\varepsilon\log E\left( \hat{\eta }%
^{\varepsilon}\left( \mathbf{X}^{\varepsilon};\boldsymbol{\alpha}\right)
\right) ^{2}
\end{equation*}
and a lower bound for \textbf{\ }%
\begin{equation*}
\liminf_{\varepsilon\rightarrow0}-\varepsilon\log E\left( \hat{\eta }%
^{\varepsilon}\left( \mathbf{X}^{\varepsilon};\boldsymbol{\alpha}\right)
\right) ^{2}
\end{equation*}
respectively. The proof will use the next three results. The first extends
the standard Laplace principle lower bound and upper bound \cite[Section 1.2
and Theorem 1.3.6.]{dupell4}, and its proof is in the Appendix.

\begin{lemma}
\label{Lem1}Given a metric space $\mathcal{X}$, suppose $\left\{
X^{\varepsilon}\right\}  \subset\mathcal{X}$ satisfies the LDP with rate
function $I:\mathcal{X}\rightarrow\lbrack0,\infty]$. Then the following
conclusions hold.

\begin{enumerate}
\item For any lower semi-continuous function $f:\mathcal{X}\rightarrow
\mathbb{R}$ that is bounded below and any closed set $B\subset$ $\mathcal{X}%
$,
\[
\liminf_{\varepsilon\rightarrow0}-\varepsilon\log E\left(  e^{-\frac
{1}{\varepsilon}f\left(  X^{\varepsilon}\right)  }1_{B}\left(  X^{\varepsilon
}\right)  \right)  \geq\inf_{x\in B}\left[  f\left(  x\right)  +I\left(
x\right)  \right]  .
\]

\item Suppose that $\left\{  f_{\ell}\right\}  _{\ell=1,\ldots,N}$ are
continuous and $\left\{  g_{\ell}\right\}  _{\ell=1,2}$ are lower
semi-continuous. Suppose also that $g_{r}(x)-\min_{\ell\in\left\{
1,\ldots,N\right\}  }\left\{  f_{\ell}(x)\right\} $ is bounded below for $r=1,2$ and that
$B_{1},B_{2}\subset$ $\mathcal{X}$ are closed. Then
\begin{align*}
&  \liminf_{\varepsilon\rightarrow0}-\varepsilon\log E\left(  \frac
{e^{-\frac{1}{\varepsilon}\left(  g_{1}\left(  X^{\varepsilon}\right)
+g_{2}\left(  X^{\varepsilon}\right)  \right)  }}{\left(  \sum_{\ell=1}%
^{N}e^{-\frac{1}{\varepsilon}f_{\ell}\left(  X^{\varepsilon}\right)  }\right)
^{2}}1_{B_{1}}\left(  X^{\varepsilon}\right)  1_{B_{2}}\left(  X^{\varepsilon
}\right)  \right) \\
&  \quad\geq\min_{r\in\left\{  1,2\right\}  }\left\{  \inf_{x\in B_{r}}\left[
2g_{r}(x)+I(x)-2\min_{\ell\in\left\{  1,\ldots,N\right\}  }\left\{  f_{\ell
}(x)\right\}  \right]  \right\}  .
\end{align*}

\item Suppose that $I:\mathcal{X}\rightarrow\lbrack0,\infty]$ is continuous on
its domain of finiteness. Then for any continuous function $f:\mathcal{X}%
\rightarrow\mathbb{R}$ that is bounded below and any closed set $B\subset$
$\mathcal{X}$ with $B$ the closure of its interior,
\[
\limsup_{\varepsilon\rightarrow0}-\varepsilon\log E\left(  e^{-\frac
{1}{\varepsilon}f\left(  X^{\varepsilon}\right)  }1_{B}\left(  X^{\varepsilon
}\right)  \right)  \leq\inf_{x\in B}\left[  f\left(  x\right)  +I\left(
x\right)  \right]  .
\]

\end{enumerate}
\end{lemma}

\bigskip It will be convenient to exclude the points in the state space
where the rate function equals $\infty$. The following lemma gives a
condition under which this is possible. The proof is in the Appendix.

\begin{lemma}
\label{Lem5}Given a metric space $\mathcal{X}$, suppose that $\left\{
X^{\varepsilon}\right\}  \subset\mathcal{X}$ satisfies the LDP with rate
function $I:\mathcal{X}\rightarrow\lbrack0,\infty]$, and assume that
\[
\limsup_{\varepsilon\rightarrow0}\varepsilon\log P\left(  X^{\varepsilon}%
\in\left\{  I=\infty\right\}  \right)  =-\infty.
\]
Let $\tilde{X}^{\varepsilon}\doteq X^{\varepsilon}1_{\{I(X^{\varepsilon
})<\infty\}}+x^{\ast}1_{\{I(X^{\varepsilon})=\infty\}}$, where $x^{\ast}%
\in\mathcal{X}$ is any point such that $I(x^{\ast})<\infty$. Then $\{\tilde
{X}^{\varepsilon}\}\subset\mathcal{\tilde{X}}$ satisfies the LDP with rate
function $\tilde{I}:\mathcal{\tilde{X}}\rightarrow\lbrack0,\infty),$ where
$\mathcal{\tilde{X}=X}\cap\{I<\infty\}$ and $\tilde{I}=I$ on $\mathcal{\tilde
{X}}$.
\end{lemma}

\begin{remark}
\label{Rmk1}In the case when $X^{\varepsilon}\sim\mu^{\varepsilon}$ with 
\begin{equation*}
\mu^{\varepsilon}\left( dx\right) =\frac{1}{Z^{\varepsilon}}e^{-\frac {1}{%
\varepsilon}I\left( x\right) }\gamma\left( dx\right) 
\end{equation*}
we find%
\begin{equation*}
P\left( X^{\varepsilon}\in\left\{ I=\infty\right\} \right) =\int_{\left\{
I=\infty\right\} }\frac{1}{Z^{\varepsilon}}e^{-\frac{1}{\varepsilon}I\left(
x\right) }\gamma\left( dx\right) =\int_{\left\{ I=\infty\right\}
}0\gamma\left( dx\right) =0, 
\end{equation*}
which gives%
\begin{equation*}
\limsup\limits_{\varepsilon\rightarrow0}\varepsilon\log P\left(
X^{\varepsilon}\in\left\{ I=\infty\right\} \right) =-\infty. 
\end{equation*}
It follows that under Condition \ref{con:1} we can assume without loss that $%
I(x)<\infty$ for each $x\in\mathcal{X}$ and $I$ is continuous on $\mathcal{X}
$.
\end{remark}

\begin{lemma}
\label{Lem4}Suppose $\left\{  \mathbf{X}^{\varepsilon}\right\}  _{\varepsilon
}\subset\mathcal{X}^{K}$ satisfies the LDP with rate function $I,$ and let
$\mathbf{Y}^{\varepsilon}$ be a symmetrized version of $\mathbf{X}%
^{\varepsilon},$ that is%
\[
P\left(  \mathbf{Y}^{\varepsilon}\in A\right)  =\frac{1}{K!}\sum\nolimits_{\sigma
\in\Sigma_{K}}P\left(  \mathbf{X}_{\sigma}^{\varepsilon}\in A\right)  \text{
for all }A\subset\mathcal{B}(\mathcal{X}^{K}),
\]
where
\[
\mathbf{X}_{\sigma}^{\varepsilon}\doteq\left(  X_{\sigma\left(  1\right)
}^{\varepsilon},\ldots,X_{\sigma\left(  K\right)  }^{\varepsilon}\right)  .
\]
Then $\left\{  \mathbf{Y}^{\varepsilon}\right\}  _{\varepsilon}\subset
\mathcal{X}^{K}$ satisfies the LDP with rate function
\[
J\left(  \mathbf{x}\right)  =\min\nolimits_{\sigma\in\Sigma_{K}}I\left(  x_{\sigma
\left(  1\right)  },\ldots,x_{\sigma\left(  K\right)  }\right)  .
\]

\end{lemma}

\begin{proof}
Since $\left\{  \mathbf{X}^{\varepsilon}\right\}  _{\varepsilon}$ satisfies
the LDP with rate function $I$ the Laplace principle holds \cite[Section
1.2]{dupell4}, and therefore for any bounded and continuous function
$g:\mathcal{X}^{K}\rightarrow\mathbb{R}$%
\[
\lim_{\varepsilon\rightarrow0}\varepsilon\log Ee^{-\frac{1}{\varepsilon
}g\left(  X^{\varepsilon}\right)  }=-\inf\nolimits_{\mathbf{x}\in\mathcal{X}^{K}%
}\left[  g\left(  \mathbf{x}\right)  +I\left(  \mathbf{x}\right)  \right]  .
\]
For such $g$ and any $\sigma\in\Sigma_{K},$ define $g_{\sigma}:\mathcal{X}%
^{K}\rightarrow\mathbb{R}$ by
\[
g_{\sigma}\left(  x_{1},\ldots,x_{K}\right)  =g\left(  x_{\sigma\left(
1\right)  },\ldots,x_{\sigma\left(  K\right)  }\right)  ,
\]
so that $g_{\sigma}$ is also bounded and continuous. Then $g\left(
\mathbf{X}_{\sigma}^{\varepsilon}\right)  =g_{\sigma}\left(  \mathbf{X}%
^{\varepsilon}\right)  $, and therefore
\[
\lim_{\varepsilon\rightarrow0}\varepsilon\log Ee^{-\frac{1}{\varepsilon
}g\left(  \mathbf{X}_{\sigma}^{\varepsilon}\right)  }=\lim_{\varepsilon
\rightarrow0}\varepsilon\log Ee^{-\frac{1}{\varepsilon}g_{\sigma}\left(
\mathbf{X}^{\varepsilon}\right)  }=-\inf_{\mathbf{x}\in\mathcal{X}^{K}}\left[
g_{\sigma}\left(  \mathbf{x}\right)  +I\left(  \mathbf{x}\right)  \right]  .
\]
Since $\mathbf{Y}^{\varepsilon}$ is a symmetrized version of $\mathbf{X}%
^{\varepsilon}$%
\[
\varepsilon\log Ee^{-\frac{1}{\varepsilon}g\left(  \mathbf{Y}^{\varepsilon
}\right)  }=\varepsilon\log\left[  \frac{1}{K!}\sum\nolimits_{\sigma\in\Sigma_{K}%
}Ee^{-\frac{1}{\varepsilon}g\left(  \mathbf{X}_{\sigma}^{\varepsilon}\right)
}\right]  ,
\]
and thus with $I_{\sigma^{-1}}\left(  \mathbf{x}\right)  =I(\mathbf{x}%
_{\sigma^{-1}})$
\begin{align*}
 \lim\nolimits_{\varepsilon\rightarrow0}\varepsilon\log Ee^{-\frac{1}{\varepsilon
}g\left(  \mathbf{Y}^{\varepsilon}\right)  }
&  =-\min\nolimits_{\sigma\in\Sigma_{K}}\left\{  \inf\nolimits_{\mathbf{x}\in
\mathcal{X}^{K}}\left[  g_{\sigma}\left(  \mathbf{x}\right)  +I\left(
\mathbf{x}\right)  \right]  \right\} \\
&  =-\min\nolimits_{\sigma\in\Sigma_{K}}\left\{  \inf\nolimits_{\mathbf{x}\in
\mathcal{X}^{K}}\left[  g\left(  \mathbf{x}\right)  +I_{\sigma^{-1}}\left(
\mathbf{x}\right)  \right]  \right\} \\
&  =-\inf\nolimits_{\mathbf{x}\in\mathcal{X}^{K}}\left[  g\left(  \mathbf{x}%
\right)  +\min\nolimits_{\sigma\in\Sigma_{K}}I_{\sigma^{-1}}\left(  \mathbf{x}\right)
\right] \\
&  =-\inf\nolimits_{\mathbf{x}\in\mathcal{X}^{K}}\left[  g\left(  \mathbf{x}%
\right)  +J\left(  \mathbf{x}\right)  \right]  .
\end{align*}
\medskip
\end{proof}

\begin{proof}
[Proof of \textbf{Theorem \ref{Thm1}}]Recall that $\hat{\eta}^{\varepsilon
}\left(  \mathbf{X}^{\varepsilon};\boldsymbol{\alpha}\right)  \overset{d}%
{=}\hat{\eta}^{\varepsilon}\left(  \mathbf{Y}^{\varepsilon};\boldsymbol{\alpha
}\right)  $, and in particular%
\[
E\left(  \hat{\eta}^{\varepsilon}\left(  \mathbf{X}^{\varepsilon
};\boldsymbol{\alpha}\right)  \right)  ^{2}=E\left(  \hat{\eta}^{\varepsilon
}\left(  \mathbf{Y}^{\varepsilon};\boldsymbol{\alpha}\right)  \right)  ^{2},
\]
so the decay rate of $E(\hat{\eta}^{\varepsilon}\left(  \mathbf{X}%
^{\varepsilon};\boldsymbol{\alpha}\right)  )^{2}$ is the same as that of
$E(\hat{\eta}^{\varepsilon}\left(  \mathbf{Y}^{\varepsilon};\boldsymbol{\alpha
}\right)  )^{2}$. By definition,
\begin{align*}
 E\left(  \hat{\eta}^{\varepsilon}\left(  \mathbf{Y}^{\varepsilon
};\boldsymbol{\alpha}\right)  \right)  ^{2}
&  =E\left(  \sum\nolimits_{\sigma\in\Sigma_{K}}\rho^{\varepsilon}\left(
\mathbf{Y}_{\sigma}^{\varepsilon};\boldsymbol{\alpha}\right)  e^{-\frac
{1}{\varepsilon}F\left(  Y_{\sigma\left(  1\right)  }^{\varepsilon}\right)
}\right)  ^{2}\\
&  =\sum\nolimits_{\sigma\in\Sigma_{K}}\sum\nolimits_{\bar{\sigma}\in\Sigma_{K}}E\left(
\rho^{\varepsilon}\left(  \mathbf{Y}_{\sigma}^{\varepsilon};\boldsymbol{\alpha
}\right)  \rho^{\varepsilon}\left(  \mathbf{Y}_{\bar{\sigma}}^{\varepsilon
};\boldsymbol{\alpha}\right)  e^{-\frac{1}{\varepsilon}F\left(  Y_{\sigma
\left(  1\right)  }^{\varepsilon}\right)  }e^{-\frac{1}{\varepsilon}F\left(
Y_{\bar{\sigma}\left(  1\right)  }^{\varepsilon}\right)  }\right)  .
\end{align*}
According to Remark \ref{Rmk1}, we may assume that $I$ is finite and continuous.

For $\tau\in\Sigma_{K}$ define $f_{\tau}\left(  \mathbf{x}\right)  \doteq
\sum_{j=1}^{K}\alpha_{j}I\left(  x_{\tau\left(  j\right)  }\right)  $, and for
$\sigma,\bar{\sigma}\in\Sigma_{K}$ let%
\[
g_{1}\left(  \mathbf{x}\right)  \doteq f_{\sigma}\left(  \mathbf{x}\right)
+F\left(  x_{\sigma\left(  1\right)  }\right)  \text{ and }g_{2}\left(
\mathbf{x}\right)  \doteq f_{\bar{\sigma}}\left(  \mathbf{x}\right)  +F\left(
x_{\bar{\sigma}\left(  1\right)  }\right)  .
\]
Since $I$ and $F$ are finite, continuous and bounded below, $g_{r}$ and
$f_{\tau}$ are also finite, continuous and  $g_{r}(x)-\min_{\ell\in\left\{
1,\ldots,N\right\}  }\left\{  f_{\ell}(x)\right\} $ is bounded below for $r=1,2$.

Applying Lemma \ref{Lem1} with $B_{1},$ $B_{2}=\mathcal{X}^{K}$,
Lemma \ref{Lem4}, and using the definition of $\rho^{\varepsilon}$ in
(\ref{eqn:defofrho}) gives%
\begin{align*}
&  \liminf_{\varepsilon\rightarrow0}-\varepsilon\log E\left(  \rho
^{\varepsilon}\left(  \mathbf{Y}_{\sigma}^{\varepsilon};\boldsymbol{\alpha
}\right)  \rho^{\varepsilon}\left(  \mathbf{Y}_{\bar{\sigma}}^{\varepsilon
};\boldsymbol{\alpha}\right)  e^{-\frac{1}{\varepsilon}F\left(  Y_{\sigma
\left(  1\right)  }^{\varepsilon}\right)  }e^{-\frac{1}{\varepsilon}F\left(
Y_{\bar{\sigma}\left(  1\right)  }^{\varepsilon}\right)  }\right) \\
&  \quad\geq\inf_{\mathbf{x\in}\mathcal{X}^{K}}\left[  2f_{\sigma}\left(
\mathbf{x}\right)  +2F\left(  x_{\sigma\left(  1\right)  }\right)  -\min
_{\tau\in\Sigma_{K}}\left\{  f_{\tau}\left(  \mathbf{x}\right)  \right\}
\right] 
{\textstyle\bigwedge}
\inf_{\mathbf{x\in}\mathcal{X}^{K}}\left[  2f_{\bar{\sigma}}\left(
\mathbf{x}\right)  +2F\left(  x_{\bar{\sigma}\left(  1\right)  }\right)
-\min_{\tau\in\Sigma_{K}}\left\{  f_{\tau}\left(  \mathbf{x}\right)  \right\}
\right]  .
\end{align*}
Let $\iota\in\Sigma_{K}$ denote the identity permutation. Since for any
nonnegative sequences $\left\{  a_{\varepsilon}\right\}  $ and $\left\{
b_{\varepsilon}\right\}  $
\[
\liminf_{\varepsilon\rightarrow0}-\varepsilon\log\left(  a_{\varepsilon
}+b_{\varepsilon}\right)  =\min\left\{  \liminf_{\varepsilon\rightarrow
0}-\varepsilon\log a_{\varepsilon},\liminf_{\varepsilon\rightarrow
0}-\varepsilon\log b_{\varepsilon}\right\}  ,
\]
we find that%
\begin{align*}
 \liminf_{\varepsilon\rightarrow0}-\varepsilon\log E\left(  \hat{\eta
}^{\varepsilon}\left(  \mathbf{Y}^{\varepsilon};\boldsymbol{\alpha}\right)
\right)  ^{2}
&  \geq\min_{\sigma\in\Sigma_{K}}\left\{  \inf_{\mathbf{x\in}%
\mathcal{X}^{K}}\left[  2f_{\sigma}\left(  \mathbf{x}\right)  +2F\left(
x_{\sigma\left(  1\right)  }\right)  -\min_{\tau\in\Sigma_{K}}\left\{
f_{\tau}\left(  \mathbf{x}\right)  \right\}  \right]  \right\} \\
&  =\inf_{\mathbf{x\in}\mathcal{X}^{K}}\left[  2f_{\iota}\left(
\mathbf{x}\right)  +2F\left(  x_{1}\right)  -\min_{\tau\in\Sigma_{K}}\left\{
f_{\tau}\left(  \mathbf{x}\right)  \right\}  \right]  ,
\end{align*}
where the equality holds because every infimum takes the same value.

For the other direction, observe that
\[
\rho^{\varepsilon}\left(  \mathbf{x};\boldsymbol{\alpha}\right)   
=\frac{e^{-\frac{1}{\varepsilon}f_{\iota}\left(  \mathbf{x}\right)  }}%
{\sum_{\tau\in\Sigma_{K}}e^{-\frac{1}{\varepsilon}f_{\tau}\left(
\mathbf{x}\right)  }}
  \geq\frac{e^{-\frac{1}{\varepsilon}f_{\iota}\left(  \mathbf{x}\right)  }%
}{K!\max_{\tau\in\Sigma_{K}}\left(  e^{-\frac{1}{\varepsilon}f_{\tau}\left(
\mathbf{x}\right)  }\right)  }  =\frac{1}{K!}e^{-\frac{1}{\varepsilon}\left[  f_{\iota}\left(
\mathbf{x}\right)  -\min_{\tau\in\Sigma_{K}}\left(  f_{\tau}\left(
\mathbf{x}\right)  \right)  \right]  }.
\]
For any $\sigma,\bar{\sigma}\in\Sigma_{K}$, the function
\[
f_{\sigma}\left(  \mathbf{x}\right)  +f_{\bar{\sigma}}\left(  \mathbf{x}%
\right)  +F\left(  x_{\sigma\left(  1\right)  }\right)  +F\left(
x_{\bar{\sigma}\left(  1\right)  }\right)  -2\min_{\tau\in\Sigma_{K}}\left(
f_{\tau}\left(  \mathbf{x}\right)  \right)  ,
\]
is finite, continuous and bounded below. Thus applying part 3 of Lemma
\ref{Lem1} with $B=\mathcal{X}^{K}$ gives
\begin{align*}
&  \limsup_{\varepsilon\rightarrow0}-\varepsilon\log E\left[  \rho
^{\varepsilon}\left(  \mathbf{Y}_{\sigma}^{\varepsilon};\boldsymbol{\alpha
}\right)  \rho^{\varepsilon}\left(  \mathbf{Y}_{\bar{\sigma}}^{\varepsilon
};\boldsymbol{\alpha}\right)  e^{-\frac{1}{\varepsilon}F\left(  Y_{\sigma
\left(  1\right)  }^{\varepsilon}\right)  }e^{-\frac{1}{\varepsilon}F\left(
Y_{\bar{\sigma}\left(  1\right)  }^{\varepsilon}\right)  }\right] \\
&  \leq\limsup_{\varepsilon\rightarrow0}-\varepsilon\log E\left[  \exp
-\frac{1}{\varepsilon}\left[  f_{\sigma}\left(  \mathbf{Y}^{\varepsilon
}\right)  +f_{\bar{\sigma}}\left(  \mathbf{Y}^{\varepsilon}\right)  
  +F\left(
Y_{\sigma\left(  1\right)  }^{\varepsilon}\right)  +F\left(  Y_{\bar{\sigma
}\left(  1\right)  }^{\varepsilon}\right)  -2\min_{\tau}\left(  f_{\tau
}\left(  \mathbf{Y}^{\varepsilon}\right)  \right)  \right]\right] \\
&  \leq\inf_{\mathbf{x\in}\mathcal{X}^{K}}\left[  f_{\sigma}\left(
\mathbf{x}\right)  +f_{\bar{\sigma}}\left(  \mathbf{x}\right)  +F\left(
x_{\sigma\left(  1\right)  }\right)  +F\left(  x_{\bar{\sigma}\left(
1\right)  }\right)  -\min_{\tau\in\Sigma_{K}}\left(  f_{\tau}\left(
\mathbf{x}\right)  \right)  \right]  .
\end{align*}
Since for any nonnegative sequences $\left\{  a_{\varepsilon}\right\}  $ and
$\left\{  b_{\varepsilon}\right\}  $
\[
\limsup_{\varepsilon\rightarrow0}-\varepsilon\log\left(  a_{\varepsilon
}+b_{\varepsilon}\right)  \leq\min\left\{  \limsup_{\varepsilon\rightarrow
0}-\varepsilon\log a_{\varepsilon},\limsup_{\varepsilon\rightarrow
0}-\varepsilon\log b_{\varepsilon}\right\}  ,
\]
we find%
\begin{align*}
&  \limsup_{\varepsilon\rightarrow0}-\varepsilon\log E\left(  \hat{\eta
}^{\varepsilon}\left(  \mathbf{Y}^{\varepsilon};\boldsymbol{\alpha}\right)
\right)  ^{2}\\
&   \leq\min_{\sigma,\bar{\sigma}\in\Sigma_{K}}\left\{  \limsup_{\varepsilon
\rightarrow0}-\varepsilon\log E\left[  \rho^{\varepsilon}\left(
\mathbf{Y}_{\sigma}^{\varepsilon};\boldsymbol{\alpha}\right)  \rho
^{\varepsilon}\left(  \mathbf{Y}_{\bar{\sigma}}^{\varepsilon}%
;\boldsymbol{\alpha}\right)  e^{-\frac{1}{\varepsilon}F\left(  Y_{\sigma
\left(  1\right)  }^{\varepsilon}\right)  }e^{-\frac{1}{\varepsilon}F\left(
Y_{\bar{\sigma}\left(  1\right)  }^{\varepsilon}\right)  }\right]  \right\} \\
&  \leq\min_{\sigma,\bar{\sigma}\in\Sigma_{K}}\left\{  \inf_{\mathbf{x\in
}\mathcal{X}^{K}}\left[  f_{\sigma}\left(  \mathbf{x}\right)  +f_{\bar{\sigma
}}\left(  \mathbf{x}\right)  +F\left(  x_{\sigma\left(  1\right)  }\right)
+F\left(  x_{\bar{\sigma}\left(  1\right)  }\right)  -\min_{\tau\in\Sigma_{K}%
}\left(  f_{\tau}\left(  \mathbf{x}\right)  \right)  \right]  \right\} \\
&  =\min_{\sigma\in\Sigma_{K}}\left\{  \inf_{\mathbf{x\in}\mathcal{X}^{K}%
}\left[  2f_{\sigma}\left(  \mathbf{x}\right)  +2F\left(  x_{\sigma\left(
1\right)  }\right)  -\min_{\tau\in\Sigma_{K}}\left(  f_{\tau}\left(
\mathbf{x}\right)  \right)  \right]  \right\} \\
&  =\inf_{\mathbf{x\in}\mathcal{X}^{K}}\left[  2f_{\iota}\left(
\mathbf{x}\right)  +2F\left(  x_{1}\right)  -\min_{\tau\in\Sigma_{K}}\left\{
f_{\tau}\left(  \mathbf{x}\right)  \right\}  \right]  .
\end{align*}
Thus the upper and lower bounds coincide. We now use Lemma \ref{Lem2} to
identify the limit as $V_{F}(\boldsymbol{\alpha})$, which completes the
proof.\medskip
\end{proof}

\begin{proof}
[Proof of Lemma \ref{Lem2}]The first step is to express the space
$\mathcal{X}^{K}$ as $\cup_{\tau\in\Sigma_{K}}O_{\tau},$ where
\[
O_{\tau}\doteq\left\{  \mathbf{x}\in\mathcal{X}^{K}:I\left(  x_{\tau\left(
1\right)  }\right)  \leq I\left(  x_{\tau\left(  2\right)  }\right)
\leq\cdots\leq I\left(  x_{\tau\left(  K\right)  }\right)  \right\}  .
\]
For any $\tau\in\Sigma_{K}$ there exists $L\in\{1,\ldots,K\}$ which depends on
$\tau$ such that $1=\tau\left(  L\right)  .$ We will use the rearrangement
inequality \cite[Section 10.2, Theorem 368]{harlitpol}, which says that if
$\mathbf{x}\in O_{\tau},$ then since $\alpha_{j}$ is nonincreasing in $j$ the
minimum in
\[
\min_{\sigma\in\Sigma_{K}}\left\{  \sum\nolimits_{j=1}^{K}\alpha_{j}I\left(
x_{\sigma\left(  j\right)  }\right)  \right\}
\]
is at $\sigma=\tau$. Thus
\begin{align*}
&  \inf_{\mathbf{x\in}\mathcal{X}^{K}}\left[  2\sum\nolimits_{j=1}^{K}\alpha
_{j}I\left(  x_{_{j}}\right)  +2F\left(  x_{1}\right)  -\min_{\sigma\in
\Sigma_{K}}\left\{  \sum\nolimits_{j=1}^{K}\alpha_{j}I\left(  x_{\sigma\left(
j\right)  }\right)  \right\}  \right] \\
&  =\min_{\tau\in\Sigma_{K}}\left\{  \inf_{\mathbf{x}\in O_{\tau}}\left[
2\sum\nolimits_{j=1}^{K}\alpha_{j}I\left(  x_{_{j}}\right)  +2F\left(  x_{1}\right)
-\min_{\sigma\in\Sigma_{K}}\left\{  \sum\nolimits_{j=1}^{K}\alpha_{j}I\left(
x_{\sigma\left(  j\right)  }\right)  \right\}  \right]  \right\} \\
&  =\min_{\tau\in\Sigma_{K}}\left\{  \inf_{\mathbf{x}\in O_{\tau}}\left[
\sum\nolimits_{j=1}^{K}\left(  2\alpha_{\tau\left(  j\right)  }-\alpha_{j}\right)
I\left(  x_{\tau\left(  j\right)  }\right)  +2F\left(  x_{1}\right)  \right]
\right\}  .
\end{align*}

Let $\beta_{j}\doteq2\alpha_{\tau\left(  j\right)  }-\alpha_{j}$, and for each
$L\in\{1,\ldots,K\}$ define the sets
\[
O_{\tau}^{L}\doteq\left\{  \left(  x_{\tau\left(  1\right)  },\ldots
,x_{\tau\left(  L\right)  }\right)  :\mathbf{x}\in O_{\tau}\right\}
\]
and
\[
\bar{O}_{\tau}^{L}\left(  \mathbf{y}\right)  \doteq\left\{  \left(
x_{\tau\left(  L\right)  },\ldots,x_{\tau\left(  K\right)  }\right)
:\mathbf{x}\in O_{\tau}\text{ and }\left(  x_{\tau\left(  1\right)  }%
,\ldots,x_{\tau\left(  L\right)  }\right)  =\mathbf{y}\right\}  .
\]
Note that for each $\tau$ (and using that $L$ is the index such that
$\tau\left(  L\right)  =1$)%
\begin{align*}
&  \inf_{\mathbf{x}\in O_{\tau}}\left[  \sum\nolimits_{j=1}^{K}\beta_{j}I\left(
x_{\tau\left(  j\right)  }\right)  +2F\left(  x_{1}\right)  \right] \\
&  =\inf_{\mathbf{x}\in O_{\tau}}\left[  \sum\nolimits_{j=1}^{L-1}\beta_{j}I\left(
x_{\tau\left(  j\right)  }\right)  +\beta_{L}I\left(  x_{\tau\left(  L\right)
}\right)  +\sum\nolimits_{j=L+1}^{K}\beta_{j}I\left(  x_{\tau\left(  j\right)
}\right)  +2F\left(  x_{\tau\left(  L\right)  }\right)  \right] \\
&  =\inf_{\left(  y_{1},\ldots,y_{L}\right)  \in O_{\tau}^{L}}\left[
\begin{array}
[c]{c}%
\sum_{j=1}^{L-1}\beta_{j}I\left(  y_{j}\right)  +\beta_{L}I\left(
y_{L}\right)  +2F\left(  y_{L}\right) \\
+\inf_{\left(  z_{L},\ldots,z_{K}\right)  \in\bar{O}_{\tau}^{L}\left(
y_{1},\ldots,y_{L}\right)  }\left[  \sum_{j=L+1}^{K}\beta_{j}I\left(
z_{j}\right)  \right]
\end{array}
\right]  .
\end{align*}

Next we show that given $\left(  y_{1},\ldots,y_{L}\right)  $ (and noting that
by definition $z_{L}=y_{L}$),
\begin{equation}
\inf_{\left(  z_{L},\ldots,z_{K}\right)  \in\bar{O}_{\tau}^{L}\left(
y_{1},\ldots,y_{L}\right)  }\left[  \sum\nolimits_{j=L+1}^{K}\beta_{j}I\left(
z_{j}\right)  \right]  =\left(  \sum\nolimits_{j=L+1}^{K}\beta_{j}\right)  I\left(
y_{L}\right)  . \label{eqn:parconst}%
\end{equation}
Recall that $\alpha_{1}\geq\alpha_{2}\geq\cdots\geq\alpha_{K}\geq0$. Therefore
$\beta_{K}=2\alpha_{\tau\left(  K\right)  }-\alpha_{K}\geq2\alpha_{K}%
-\alpha_{K}=\alpha_{K}\geq0$. More generally, since $\tau(j),\ldots,\tau(K)$
are distinct values drawn from $\{1,\ldots,K\}$, for each $j$%
\[
\beta_{j}+\cdots+\beta_{K}=2\sum\nolimits_{\ell=j}^{K}\alpha_{\tau\left(  \ell\right)
}-\sum\nolimits_{\ell=j}^{K}\alpha_{\ell}\geq2\sum\nolimits_{\ell=j}^{K}\alpha_{\ell}-\sum\nolimits
_{\ell=j}^{K}\alpha_{\ell}\geq0.
\]
Using $\beta_{K}\geq0$ and the fact that $\left(  z_{L},\ldots,z_{K}\right)
\in\bar{O}_{\tau}^{L}\left(  y_{1},\ldots,y_{L}\right)  $ implies the
restriction%
\[
I\left(  z_{L}\right)  \leq I\left(  z_{L+1}\right)  \leq\cdots\leq I\left(
z_{K}\right)  ,
\]
we can rewrite the infimum as
\begin{align*}
&  \inf_{\left(  z_{L},\ldots,z_{K}\right)  \in\bar{O}_{\tau}^{L}\left(
y_{1},\ldots,y_{L}\right)  }\left[  \sum\nolimits_{j=L+1}^{K}\beta_{j}I\left(
z_{j}\right)  \right] \\
&  \quad=\inf_{\left(  z_{L},\ldots,z_{K}\right)  \in\bar{O}_{\tau}^{L}\left(
y_{1},\ldots,y_{L}\right)  }\left[  \sum\nolimits_{j=L+1}^{K-2}\beta_{j}I\left(
z_{j}\right)  +\left(  \beta_{K-1}+\beta_{K}\right)  I\left(  z_{K-1}\right)
\right]  .
\end{align*}
Iterating, we have (\ref{eqn:parconst}). Hence recalling $D\doteq
\{(I(x),F(x)):x\in\mathcal{X}\}$,%
\begin{align*}
&  \inf_{\mathbf{x\in}\mathcal{X}^{K}}\left[  2\sum\nolimits_{j=1}^{K}\alpha
_{j}I\left(  x_{_{j}}\right)  +2F\left(  x_{1}\right)  -\min_{\sigma\in
\Sigma_{K}}\left\{  \sum\nolimits_{j=1}^{K}\alpha_{j}I\left(  x_{\sigma\left(
j\right)  }\right)  \right\}  \right] \\
&  =\min_{\tau\in\Sigma_{K}}\left\{  \inf_{\mathbf{x}\in O_{\tau}}\left[
\sum\nolimits_{j=1}^{K}\left(  2\alpha_{\tau\left(  j\right)  }-\alpha_{j}\right)
I\left(  x_{\tau\left(  j\right)  }\right)  +2F\left(  x_{1}\right)  \right]
\right\} \\
&  =\min_{\tau\in\Sigma_{K}}\left\{  \inf_{\left(  x_{\tau\left(  1\right)
},\ldots,x_{\tau\left(  L\right)  }\right)  \in O_{\tau}^{L}}\left[
\sum\nolimits_{j=1}^{L-1}\beta_{j}I\left(  x_{\tau\left(  j\right)  }\right)  +\left(
\sum\nolimits_{j=L}^{K}\beta_{j}\right)  I\left(  x_{\tau\left(  L\right)  }\right)
+2F\left(  x_{\tau\left(  L\right)  }\right)  \right]  \right\} \\
&  =\min_{\tau\in\Sigma_{K}}\left\{  \inf_{\substack{\left\{  \left(
I_{\tau\left(  L\right)  },F\right)  :\left(  I_{\tau\left(  L\right)
},F\right)  \in D\right\}  \\\left\{  \left(  I_{\tau\left(  1\right)
},\ldots,I_{\tau\left(  L-1\right)  }\right)  :I_{\tau\left(  1\right)  }\leq
I_{\tau\left(  2\right)  }\leq\cdots\leq I_{\tau\left(  L\right)  }\right\}
}}\left[  \sum\nolimits_{j=1}^{L-1}\beta_{j}I_{\tau\left(  j\right)  }+\left(
\sum\nolimits_{j=L}^{K}\beta_{j}\right)  I_{\tau\left(  L\right)  }+2F\right]
\right\}  .
\end{align*}
The last equality holds because $I$ is continuous on the domain of finiteness.

It remains to show that the last display coincides with $V_{F}\left(
\boldsymbol{\alpha}\right)  $. Recalling the definition of $V_{F}\left(
\boldsymbol{\alpha}\right)  $, we have
\begin{align*}
V_{F}\left(  \boldsymbol{\alpha}\right)   &  \doteq\inf_{\substack{\left(
I_{1},F\right)  \in D\\\left\{  (I_{1},\ldots,I_{K}):I_{j}\in\lbrack
0,I_{1}]\text{ for }j\geq2\right\}  }}\left[  2\sum\nolimits_{j=1}^{K}\alpha_{j}%
I_{j}+2F-\min_{\sigma\in\Sigma_{K}}\left\{  \sum\nolimits_{j=1}^{K}\alpha_{j}%
I_{\sigma\left(  j\right)  }\right\}  \right] \\
&  =\min_{\tau\in\Sigma_{K}}\left\{  \inf_{_{\substack{\left\{  \left(
I_{\tau\left(  L\right)  },F\right)  :\left(  I_{\tau\left(  L\right)
},F\right)  \in D\right\}  \\\left\{  \left(  I_{\tau\left(  1\right)
},\cdots,I_{\tau\left(  K\right)  }\right)  :I_{\tau\left(  1\right)  }\leq
I_{\tau\left(  2\right)  }\leq\cdots\leq I_{\tau\left(  K\right)  }\leq
I_{\tau\left(  L\right)  }\right\}  }}}\left[  \sum\nolimits_{j=1}^{K}\left(
2\alpha_{\tau\left(  j\right)  }-\alpha_{j}\right)  I_{\tau\left(  j\right)
}+2F\right]  \right\}  .
\end{align*}
Since $\boldsymbol{I}\in O_{\tau}$ implies $I_{\tau\left(  j\right)  }\geq
I_{\tau\left(  L\right)  }$ and hence $I_{\tau\left(  j\right)  }%
=I_{\tau\left(  L\right)  }$ for $L<j\leq K$,%
\begin{align*}
&  \inf_{_{_{\substack{\left\{  \left(  I_{\tau\left(  L\right)  },F\right)
:\left(  I_{\tau\left(  L\right)  },F\right)  \in D\right\}  \\\left\{
\left(  I_{\tau\left(  1\right)  },\ldots,I_{\tau\left(  K\right)  }\right)
:I_{\tau\left(  1\right)  }\leq I_{\tau\left(  2\right)  }\leq\cdots\leq
I_{\tau\left(  K\right)  }\leq I_{\tau\left(  L\right)  }\right\}  }}}}\left[
\sum\nolimits_{j=1}^{K}\beta_{j}I_{\tau\left(  j\right)  }+2F\right] \\
&  \quad=\inf_{_{\substack{\left\{  \left(  I_{\tau\left(  L\right)
},F\right)  :\left(  I_{\tau\left(  L\right)  },F\right)  \in D\right\}
\\\left\{  \left(  I_{\tau\left(  1\right)  },\ldots,I_{\tau\left(
L-1\right)  }\right)  :I_{\tau\left(  1\right)  }\leq I_{\tau\left(  2\right)
}\leq\cdots\leq I_{\tau\left(  L\right)  }\right\}  }}}\left[  \sum\nolimits
_{j=1}^{L-1}\beta_{j}I_{\tau\left(  j\right)  }+\left(  \sum\nolimits_{j=L}^{K}%
\beta_{j}\right)  I_{\tau\left(  L\right)  }+2F\right]  .
\end{align*}
This completes the proof of the first part.

To prove the second part, i.e., $\sup_{\boldsymbol{\alpha}}V_{F}\left(
\boldsymbol{\alpha}\right)  =\inf_{x}\{(2-\left(  1/2\right)  ^{K-1})I\left(
x\right)  +2F\left(  x\right)  \},$ first rewrite $V_{F}\left(
\boldsymbol{\alpha}\right)  $ by noticing that since $I_{1}$ is the largest
value in the set $\boldsymbol{I}$,%
\[
\min_{\tau\in\Sigma_{K}}\left\{  \sum\nolimits_{j=1}^{K}\alpha_{j}I_{\tau\left(
j\right)  }\right\}
\]
obtains the minimum at some $\tau\in\Sigma_{K}$ with $\tau\left(  K\right)
=1$. Therefore
\[
  V_{F}\left(  \boldsymbol{\alpha}\right) 
  =\inf_{\substack{\left(  I_{1},F\right)  \in D\\\left\{  \mathbf{I}%
:I_{j}\leq I_{1}\text{ for }j\geq2\right\}  }}\left[  \left(  2\alpha
_{1}-\alpha_{K}\right)  I_{1}+2\sum_{j=2}^{K}\alpha_{j}I_{j}+2F-\min_{\tau
\in\Sigma_{K},\tau\left(  K\right)  =1}\left\{  \sum_{j=1}^{K-1}\alpha
_{j}I_{\tau\left(  j\right)  }\right\}  \right]  .
\]
Suppose we are given any $K-1$ numbers and assign them to $\left\{
I_{j}\right\}  _{j=2,\ldots,K}$ in a certain order. Then the value of%
\[
\min_{\tau\in\Sigma_{K},\tau\left(  K\right)  =1}\left\{  \sum\nolimits_{j=1}%
^{K-1}\alpha_{j}I_{\tau\left(  j\right)  }\right\}
\]
is independent of the order. But since $\alpha_{1}\geq\cdots\geq\alpha_{K}%
\geq0$, by the rearrangement inequality the smallest value of%
\[
\sum\nolimits_{j=2}^{K}\alpha_{j}I_{j}%
\]
is obtained by taking the $I_{j},$ $j\geq2$ in increasing order. By choosing
this ordering of $\left\{  I_{j}\right\}  _{j=2,\ldots,K}$,
\[
\min_{\tau\in\Sigma_{K},\tau\left(  K\right)  =1}\left\{  \sum\nolimits_{j=1}%
^{K-1}\alpha_{j}I_{\tau\left(  j\right)  }\right\}  =\sum\nolimits_{j=2}^{K}%
\alpha_{j-1}I_{j}.
\]
Thus%
\begin{align}
 V_{F}\left(  \boldsymbol{\alpha}\right) \label{eqn:Vaexp}
&  =\inf_{\substack{\left(  I_{1},F\right)  \in D\\\left\{  \mathbf{I}:0\leq
I_{2}\leq\cdots\leq I_{K}\leq I_{1}\right\}  }}\left[  \left(  2\alpha
_{1}-\alpha_{K}\right)  I_{1}+2\sum\nolimits_{j=2}^{K}\alpha_{j}I_{j}+2F-\sum\nolimits_{j=2}%
^{K}\alpha_{j-1}I_{j}\right] \\
&  =\inf_{\substack{\left(  I_{1},F\right)  \in D\\\left\{  \mathbf{I}:0\leq
I_{2}\leq\cdots\leq I_{K}\leq I_{1}\right\}  }}\left[  \left(  2\alpha
_{1}-\alpha_{K}\right)  I_{1}+\sum\nolimits_{j=2}^{K}\left(  2\alpha_{j}-\alpha
_{j-1}\right)  I_{j}+2F\right]  .\nonumber
\end{align}
Using summation by parts and $\alpha_{1}=1$, we have
\begin{align}
V_{F}\left(  \boldsymbol{\alpha}\right) \label{eqn:Vaexp2}
&  =\inf_{\substack{\left(  I_{1},F\right)  \in D\\\left\{  \mathbf{I}:0\leq
I_{2}\leq\cdots\leq I_{K}\leq I_{1}\right\}  }}\left[  \left(  2\alpha
_{1}-\alpha_{K}\right)  I_{1}+\sum\nolimits_{j=2}^{K-1}\alpha_{j}\left(  2I_{j}%
-I_{j+1}\right)  +2\alpha_{K}I_{K}-I_{2}+2F\right]  .
\end{align}
Since $I$ is a rate function and $F$ is continuous and bounded below, there is
$\left(  I_{0},F_{0}\right)  \in D$ such that
\[
\left(  2-\left(  1/2\right)  ^{K-1}\right)  I_{0}+2F_{0}=\inf_{x}\left[
\left(  2-\left(  1/2\right)  ^{K-1}\right)  I\left(  x\right)  +2F\left(
x\right)  \right]  .
\]
Let $\boldsymbol{\alpha}^{\ast}\doteq\left(  1,1/2,\ldots,1/2^{K-1}\right)  $
and $\mathbf{I}^{\ast}=\left(  I_{1}^{\ast},\ldots,I_{K}^{\ast}\right)  ,$
with $I_{1}^{\ast}\doteq I_{0},$ $I_{j}^{\ast}\doteq\left(  1/2\right)
^{K-j+1}I_{0}$ for $j=2,\ldots,K$. We have the following inequalities, which
are explained after the display:
\begin{align*}
\left(  2-\left(  1/2\right)  ^{K-1}\right)  I_{0}+2F_{0}
&  =\inf_{\substack{\left(  I_{1},F\right)  \in D\\\left\{  \mathbf{I}:0\leq
I_{2}\leq\cdots\leq I_{K}\leq I_{1}\right\}  }}\left[  \left(  2\alpha
_{1}^{\ast}-\alpha_{K}^{\ast}\right)  I_{1}+\sum\nolimits_{j=2}^{K}\left(  2\alpha
_{j}^{\ast}-\alpha_{j-1}^{\ast}\right)  I_{j}+2F\right] \\
&  =V_{F}\left(  \boldsymbol{\alpha}^{\ast}\right) \\
&  \leq\sup_{\boldsymbol{\alpha}}V_{F}\left(  \boldsymbol{\alpha}\right) \\
&  \leq\sup_{\boldsymbol{\alpha}}\left[  \left(  2\alpha_{1}-\alpha
_{K}\right)  I_{1}^{\ast}+\sum\nolimits_{j=2}^{K-1}\alpha_{j}\left(  2I_{j}^{\ast
}-I_{j+1}^{\ast}\right)  +2\alpha_{K}I_{K}^{\ast}-I_{2}^{\ast}+2F_{0}\right]
\\
&  =\left(  2-\left(  1/2\right)  ^{K-1}\right)  I_{0}+2F_{0}.
\end{align*}
The first equality follows from $2\alpha_{j}^{\ast}-\alpha_{j-1}^{\ast}=0$ for
$j=2,\ldots,K$; the second equality from (\ref{eqn:Vaexp}); the second
inequality is from (\ref{eqn:Vaexp2}); the third equality uses $\alpha_{1}=1$,
$2I_{j}^{\ast}-I_{j+1}^{\ast}=0$ for $j=2,\ldots,K$, $-$ $\alpha_{K}%
I_{1}^{\ast}+2\alpha_{K}I_{K}^{\ast}=0$ and $I_{2}^{\ast}=\left(  1/2\right)
^{K-1}I_{0}$. We therefore obtain%
\[
\sup_{\boldsymbol{\alpha}}V_{F}\left(  \boldsymbol{\alpha}\right)  =\inf
_{x}\left\{  \left(  2-\left(  1/2\right)  ^{K-1}\right)  I\left(
x\right)  +2F\left(  x\right)  \right\}  .
\]

\end{proof}

\begin{remark}
Note that the proof identifies $(\boldsymbol{\alpha}^{\ast},\mathbf{I}^{\ast
})$\ as a saddle point of a min/max problem.
\end{remark}

\subsection{Estimating the probability of rare events}

\label{sec:estimating_rare_events} We next state the corresponding results
for the INS estimator for approximating probabilities as in (\ref%
{eqn:setAprob}). The statements are what one would obtain by letting $F$
approximate $-M1_{A^{c}}$ and sending $M\rightarrow\infty$. The proofs are
similar to but easier than those of the risk-sensitive functional, and thus
omitted.

\begin{definition}
Given any $K\in\mathbb{N}$ and $\boldsymbol{\alpha}\mathbf{,}$ let $%
\{X_{j}^{\varepsilon}\}$ be independent with $X_{j}^{\varepsilon}\sim
\mu^{\alpha_{j}/\varepsilon}$ for $j\in\left\{ 1,\ldots,K\right\} $. With
the weights $\rho^{\varepsilon}\left( \mathbf{x};\boldsymbol{\alpha}\right) $
defined according to (\ref{eqn:defofrho}), the $K$ temperature INS estimator
for estimating $\mu^{\varepsilon}\left( A\right) $ is 
\begin{equation*}
\hat{\theta}^{\varepsilon}\left( \mathbf{X}^{\varepsilon};\boldsymbol{\alpha 
}\right) =\sum\nolimits_{\sigma\in\Sigma_{K}}\rho^{\varepsilon}\left( 
\mathbf{X}_{\sigma}^{\varepsilon};\boldsymbol{\alpha}\right)
1_{A}(X_{\sigma\left( 1\right) }^{\varepsilon}). 
\end{equation*}
\end{definition}

Following the same argument as for $\hat{\eta}^{\varepsilon}\left( \mathbf{X}%
^{\varepsilon};\boldsymbol{\alpha}\right) $ but replacing the function $e^{-%
\frac{1}{\varepsilon}F\left( x\right) }$ with $1_{A}\left( x\right) $ gives

\begin{itemize}
\item $\hat{\theta}^{\varepsilon}\left( \mathbf{X}^{\varepsilon };%
\boldsymbol{\alpha}\right) $ is an unbiased estimator for $\mu
^{\varepsilon}\left( A\right) $,

\item $\hat{\theta}^{\varepsilon}\left( \mathbf{X}^{\varepsilon };%
\boldsymbol{\alpha}\right) \overset{d}{=}\hat{\theta}^{\varepsilon}\left( 
\mathbf{Y}^{\varepsilon};\boldsymbol{\alpha}\right) $.
\end{itemize}

\begin{theorem}
\label{Thm2} Let $A$ be the closure of its interior. Suppose that $I:%
\mathcal{X}\rightarrow\lbrack0,\infty]$ is continuous on its domain of
finiteness. Then for any $K\in\mathbb{N}$ and parameter $\boldsymbol{\alpha}$
\begin{equation*}
\lim_{\varepsilon\rightarrow0}-\varepsilon\log E\left( \hat{\theta }%
^{\varepsilon}\left( \mathbf{X}^{\varepsilon};\boldsymbol{\alpha}\right)
\right) ^{2}=V\left( \boldsymbol{\alpha}\right) \cdot I\left( A\right) , 
\end{equation*}
where 
\begin{equation*}
V\left( \boldsymbol{\alpha}\right) =\inf_{\left\{ I:I_{1}=1,I_{j}\in\left[
0,1\right] \text{ for }j=2,\ldots,K\right\} }\left[ 2\sum\nolimits_{j=1}^{K}%
\alpha_{j}I_{j}-\min_{\sigma\in\Sigma_{K}}\left\{
\sum\nolimits_{j=1}^{K}\alpha _{j}I_{\sigma\left( j\right) }\right\} \right]
. 
\end{equation*}
Moreover $V\left( \boldsymbol{\alpha}\right) $ achieves its maximal value of 
$2-\left( 1/2\right) ^{K}$ at $\boldsymbol{\alpha}^{\ast}=\left(
1,1/2,\ldots,1/2^{K-1}\right) $.
\end{theorem}

\section{Extensions}

\label{sec:extensions}

The theory of Section \ref{sec:performance} assumes that $\mu^{\varepsilon
}(dx)=\frac{1}{Z^{\varepsilon}}e^{-\frac{1}{\varepsilon}I(x)}\gamma(dx)$ and
that $F$ does not depend on $\varepsilon$. Suppose we know only that $%
\left\{ \mu^{\varepsilon}\right\} $ satisfies the LDP with rate function $I,$
and $\mu^{\varepsilon}$ has a density function $g^{\varepsilon}\left(
x\right) $ for any $\varepsilon>0.$ Define a single sample of the $K$
temperature INS estimator based on $\boldsymbol{\alpha}$ for estimating $%
\int_{\mathcal{X}}e^{-\frac{1}{\varepsilon}F^{\varepsilon}(x)}\mu^{%
\varepsilon}(dx)$ to be 
\begin{equation*}
\tilde{\eta}^{\varepsilon}\left( \mathbf{X}^{\varepsilon};\boldsymbol{\alpha 
}\right) \doteq\sum\nolimits_{\sigma\in\Sigma_{K}}\tilde{\rho}%
^{\varepsilon}\left( \mathbf{X}_{\sigma}^{\varepsilon};\boldsymbol{\alpha}%
\right) e^{-\frac {1}{\varepsilon}F^{\varepsilon}\left( X_{\sigma\left(
1\right) }^{\varepsilon}\right) }, 
\end{equation*}
where 
\begin{equation*}
\tilde{\rho}^{\varepsilon}\left( \mathbf{x};\boldsymbol{\alpha}\right) \doteq%
\frac{\prod_{j=1}^{K}g^{\varepsilon/\alpha_{j}}\left( x_{j}\right) }{%
\sum_{\tau\in\Sigma_{K}}\left[ \prod_{j=1}^{K}g^{\varepsilon/\alpha_{j}}%
\left( x_{\tau\left( j\right) }\right) \right] }. 
\end{equation*}
As before, one can prove that the estimator is unbiased and if $\mathbf{Y}%
^{\varepsilon}$ is a symmetrized version of $\mathbf{X}^{\varepsilon},$ then%
\begin{equation*}
\tilde{\eta}^{\varepsilon}\left( \mathbf{X}^{\varepsilon};\boldsymbol{\alpha 
}\right) \overset{d}{=}\tilde{\eta}^{\varepsilon}\left( \mathbf{Y}%
^{\varepsilon};\boldsymbol{\alpha}\right) . 
\end{equation*}
Moreover, recall that%
\begin{equation*}
\rho^{\varepsilon}\left( \mathbf{x};\boldsymbol{\alpha}\right) \doteq \frac{%
\prod_{j=1}^{K}e^{-\frac{\alpha_{j}}{\varepsilon}I\left( x_{j}\right) }}{%
\sum_{\tau\in\Sigma_{K}}\left[ \prod_{j=1}^{K}e^{-\frac{\alpha_{j}}{%
\varepsilon}I\left( x_{\tau\left( j\right) }\right) }\right] }. 
\end{equation*}

\begin{theorem}
\label{Thm3}Let $F:\mathcal{X}\rightarrow\mathbb{R}$ be continuous and
bounded below, and let $I:\mathcal{X}\rightarrow\lbrack0,\infty]$ be
continuous on its domain of finiteness. Suppose that $\varepsilon\log(\tilde{%
\rho}^{\varepsilon }/\rho^{\varepsilon})\rightarrow0$ and $%
F^{\varepsilon}\rightarrow F$ uniformly on each compact set in $\mathcal{X}%
^{K}$ and $\mathcal{X}$ respectively, and that $\inf_{\varepsilon\in(0,1)}%
\inf_{x\in\mathcal{X}}F^{\varepsilon}(x)>-\infty$. Then 
\begin{equation*}
\lim_{\varepsilon\rightarrow0}-\varepsilon\log E\left( \tilde{\eta }%
^{\varepsilon}\left( \mathbf{X}^{\varepsilon};\boldsymbol{\alpha}\right)
\right) ^{2}=V_{F}(\boldsymbol{\alpha}), 
\end{equation*}
where $V_{F}(\boldsymbol{\alpha})$ is as defined in Lemma \ref{Lem2}.
\end{theorem}

\begin{proof}
Using the uniform bounds on $F^{\varepsilon}$ and the $F$ and the fact that
$I$ is a rate function, there is a compact set $C\subset\mathcal{X}$ such
that
\[
\lim\limits_{\varepsilon\rightarrow0}\varepsilon\log E\left(  \tilde{\eta
}^{\varepsilon}\left(  \mathbf{X}^{\varepsilon};\boldsymbol{\alpha}\right)
\right)  ^{2}=\lim\limits_{\varepsilon\rightarrow0}\varepsilon\log E\left(
\tilde{\eta}^{\varepsilon}\left(  \mathbf{X}^{\varepsilon};\boldsymbol{\alpha
}\right)  1_{\left(  C\right)  ^{K}}\left(  \mathbf{X}^{\varepsilon}\right)
\right)  ^{2}%
\]
and
\[
\lim\limits_{\varepsilon\rightarrow0}\varepsilon\log E\left(  \hat{\eta
}^{\varepsilon}\left(  \mathbf{X}^{\varepsilon};\boldsymbol{\alpha}\right)
\right)  ^{2}=\lim\limits_{\varepsilon\rightarrow0}\varepsilon\log E\left(
\hat{\eta}^{\varepsilon}\left(  \mathbf{X}^{\varepsilon};\boldsymbol{\alpha
}\right)  1_{\left(  C\right)  ^{K}}\left(  \mathbf{X}^{\varepsilon}\right)
\right)  ^{2},
\]
in the sense that each limit must be equal to the other if one of them exists.
We know that
\[
\lim\limits_{\varepsilon\rightarrow0}-\varepsilon\log E\left(  \hat{\eta
}^{\varepsilon}\left(  \mathbf{X}^{\varepsilon};\boldsymbol{\alpha}\right)
\right)  ^{2}=V_{F}(\boldsymbol{\alpha}),
\]
and thus it suffices to prove
\[
\lim\limits_{\varepsilon\rightarrow0}\varepsilon\log E\left(  \tilde{\eta
}^{\varepsilon}\left(  \mathbf{X}^{\varepsilon};\boldsymbol{\alpha}\right)
1_{\left(  C\right)  ^{K}}\left(  \mathbf{X}^{\varepsilon}\right)  \right)
^{2}=\lim\limits_{\varepsilon\rightarrow0}\varepsilon\log E\left(  \hat{\eta
}^{\varepsilon}\left(  \mathbf{X}^{\varepsilon};\boldsymbol{\alpha}\right)
1_{\left(  C\right)  ^{K}}\left(  \mathbf{X}^{\varepsilon}\right)  \right)
^{2}.
\]

Using the uniform convergence of $\varepsilon\log(\tilde{\rho}^{\varepsilon
}/\rho^{\varepsilon})$ and $F^{\varepsilon}$, for any $\delta>0$ there is a
$\zeta$ such that if $0<\varepsilon<\zeta$ then
\[
\sup_{\mathbf{x}\in C^{K}}\left\vert \varepsilon\log\left(  \tilde{\rho
}^{\varepsilon}/\rho^{\varepsilon}\right)  \right\vert <\delta\text{ and }%
\sup_{x\in C}\left\vert F^{\varepsilon}\left(  x\right)  -F\left(  x\right)
\right\vert <\delta.
\]
Thus for any $\mathbf{x}\in C^{K}$ and $x\in C$
\[
\rho^{\varepsilon}\left(  \mathbf{x};\boldsymbol{\alpha}\right)  e^{-\frac
{1}{\varepsilon}\delta}<\tilde{\rho}^{\varepsilon}\left(  \mathbf{x}%
;\boldsymbol{\alpha}\right)  <\rho^{\varepsilon}\left(  \mathbf{x}%
;\boldsymbol{\alpha}\right)  e^{\frac{1}{\varepsilon}\delta}\text{ and
}F\left(  x\right)  -\delta<F^{\varepsilon}\left(  x\right)  <F\left(
x\right)  +\delta.
\]
Furthermore
\begin{align*}
\tilde{\eta}^{\varepsilon}\left(  \mathbf{X}^{\varepsilon};\boldsymbol{\alpha
}\right)   &  =\sum\nolimits_{\sigma\in\Sigma_{K}}\tilde{\rho}^{\varepsilon}\left(
\mathbf{X}_{\sigma}^{\varepsilon};\boldsymbol{\alpha}\right)  e^{-\frac
{1}{\varepsilon}F^{\varepsilon}\left(  X_{\sigma\left(  1\right)
}^{\varepsilon}\right)  }\\
&  \quad\leq\left(  \sum\nolimits_{\sigma\in\Sigma_{K}}\rho^{\varepsilon}\left(
\mathbf{X}_{\sigma}^{\varepsilon};\boldsymbol{\alpha}\right)  e^{-\frac
{1}{\varepsilon}F\left(  X_{\sigma\left(  1\right)  }^{\varepsilon}\right)
}\right)  e^{\frac{2}{\varepsilon}\delta}\\
&  \quad=\hat{\eta}^{\varepsilon}\left(  \mathbf{X}^{\varepsilon
};\boldsymbol{\alpha}\right)  e^{\frac{2\delta}{\varepsilon}}%
\end{align*}
and similarly $\tilde{\eta}^{\varepsilon}\left(  \mathbf{X}^{\varepsilon
};\boldsymbol{\alpha}\right)  \geq\hat{\eta}^{\varepsilon}\left(
\mathbf{X}^{\varepsilon};\boldsymbol{\alpha}\right)  e^{-\frac{2\delta
}{\varepsilon}}$. Putting the estimates together gives%
\[
  \limsup_{\varepsilon\rightarrow0}-\varepsilon\log E\left(  \tilde{\eta
}^{\varepsilon}\left(  \mathbf{X}^{\varepsilon};\boldsymbol{\alpha}\right)
1_{\left(  C\right)  ^{K}}\left(  \mathbf{X}^{\varepsilon}\right)  \right)
^{2}
\leq\limsup_{\varepsilon\rightarrow0}-\varepsilon\log E\left(
\hat{\eta}^{\varepsilon}\left(  \mathbf{X}^{\varepsilon};\boldsymbol{\alpha
}\right)  1_{\left(  C\right)  ^{K}}\left(  \mathbf{X}^{\varepsilon}\right)
\right)  ^{2}+2\delta
\]
and%
\[
  \liminf\limits_{\varepsilon\rightarrow0}-\varepsilon\log E\left(
\tilde{\eta}^{\varepsilon}\left(  \mathbf{X}^{\varepsilon};\boldsymbol{\alpha
}\right)  1_{\left(  C\right)  ^{K}}\left(  \mathbf{X}^{\varepsilon}\right)
\right)  ^{2} \geq\liminf\limits_{\varepsilon\rightarrow0}-\varepsilon\log E\left(
\hat{\eta}^{\varepsilon}\left(  \mathbf{X}^{\varepsilon};\boldsymbol{\alpha
}\right)  1_{\left(  C\right)  ^{K}}\left(  \mathbf{X}^{\varepsilon}\right)
\right)  ^{2}-2\delta.
\]
Since $\delta$ is arbitrary, sending $\delta\rightarrow0$ shows the existence
of
\[
\lim\limits_{\varepsilon\rightarrow0}\varepsilon\log E\left(  \tilde{\eta
}^{\varepsilon}\left(  \mathbf{X}^{\varepsilon};\boldsymbol{\alpha}\right)
1_{\left(  C\right)  ^{K}}\left(  \mathbf{X}^{\varepsilon}\right)  \right)
^{2}%
\]
and
\[
\lim\limits_{\varepsilon\rightarrow0}\varepsilon\log E\left(  \tilde{\eta
}^{\varepsilon}\left(  \mathbf{X}^{\varepsilon};\boldsymbol{\alpha}\right)
1_{\left(  C\right)  ^{K}}\left(  \mathbf{X}^{\varepsilon}\right)  \right)
^{2}=\lim\limits_{\varepsilon\rightarrow0}\varepsilon\log E\left(  \hat{\eta
}^{\varepsilon}\left(  \mathbf{X}^{\varepsilon};\boldsymbol{\alpha}\right)
1_{\left(  C\right)  ^{K}}\left(  \mathbf{X}^{\varepsilon}\right)  \right)
^{2}.
\]

\end{proof}

\begin{remark}
Theorem \ref{Thm3} is indeed an extension of Theorem \ref{Thm1}, since in
the setting of Theorem \ref{Thm1} 
\begin{equation}
g^{\varepsilon}(x)=\frac{1}{Z^{\varepsilon}}e^{-\frac{1}{\varepsilon}I(x)}, 
\label{eqn:g-form}
\end{equation}
and thus for any $\mathbf{x}$ 
\begin{equation*}
\tilde{\rho}^{\varepsilon}\left( \mathbf{x};\boldsymbol{\alpha}\right)
=\rho^{\varepsilon}\left( \mathbf{x};\boldsymbol{\alpha}\right) . 
\end{equation*}
A situation where $g^{\varepsilon}$ not of the form (\ref{eqn:g-form})
arises is with densities that are mixtures. These densities take the form $%
g^{\varepsilon}(x)=(\sum_{j=1}^{m}\omega_{j}e^{-\frac{1}{\varepsilon}%
I_{j}(x)})/Z^{\varepsilon}$, where each $I_{j}$ is a rate function and $%
\omega_{j}>0$ for all $j$. Since 
\begin{equation*}
\left( \min_{j\in\{1,\ldots,k\}}\omega_{j}\right) e^{-\frac{1}{\varepsilon }%
\min_{k}\left[ I_{k}\left( x\right) \right] }\leq\sum\nolimits_{j=1}^{m}%
\omega _{j}e^{-\frac{1}{\varepsilon}I_{j}\left( x\right) }\leq\left(
\sum\nolimits _{j=1}^{m}\omega_{j}\right) e^{-\frac{1}{\varepsilon}\min_{k}%
\left[ I_{k}\left( x\right) \right] }, 
\end{equation*}
this implies%
\begin{equation*}
\lim_{\varepsilon\rightarrow0}\varepsilon\log\left( \tilde{\rho}%
^{\varepsilon}\left( \mathbf{x};\boldsymbol{\alpha}\right) /\rho
^{\varepsilon}\left( \mathbf{x};\boldsymbol{\alpha}\right) \right) =0 
\end{equation*}
uniformly in $\mathbf{x}$.
\end{remark}

\begin{remark}
By using the allowed $\varepsilon$ dependence of $F^{\varepsilon}$, it is
possible to apply the INS estimator to certain classes of discrete random
variables. For example, let $X$ have the geometric distribution with
parameter $p,$ let $X^{\varepsilon}=\varepsilon X,$ and suppose one wants to
estimate 
\begin{equation*}
Ee^{-\frac{1}{\varepsilon}G\left( X^{\varepsilon}\right) }
\end{equation*}
with $G$ continuous and bounded below. We can represent $X^{\varepsilon}$ in
the form $X^{\varepsilon}=f^{\varepsilon}\left( Y^{\varepsilon}\right) ,$
where $Y^{\varepsilon}$ is exponentially distributed with parameter $%
\lambda/\varepsilon,$ $\lambda\doteq-\log(1-p)$ and 
\begin{equation*}
f^{\varepsilon}\left( y\right) =\varepsilon\left( \left\lfloor {y}/{%
\varepsilon}\right\rfloor +1\right) , 
\end{equation*}
and therefore consider the problem 
\begin{equation*}
Ee^{-\frac{1}{\varepsilon}F^{\varepsilon}\left( Y^{\varepsilon}\right) }, 
\end{equation*}
where $F^{\varepsilon}=G\circ f^{\varepsilon},$ and both $F^{\varepsilon}$
and $Y^{\varepsilon}$ satisfy the assumptions in Theorem \ref{Thm3}.
\end{remark}

\begin{remark}
\label{rmk:partial_ins} Suppose $\{\mu ^{\varepsilon }\}_{\varepsilon }$ are
product measures $\{\mu _{1}^{\varepsilon }\times \mu _{2}^{\varepsilon
}\}_{\varepsilon }.$ Moreover, assume that $\{\mu _{1}^{\varepsilon
}\}_{\varepsilon }$ and $\{\mu _{2}^{\varepsilon }\}_{\varepsilon }$ satisfy
Condition \ref{con:1} on Polish spaces $\mathcal{X}_{1}$ and $\mathcal{X}_{2}
$, with rate functions $I_{1}$ and $I_{2}$, respectively. In this case, we
can use INS only on subset of variables. Precisely, given any $K\in \mathbb{N%
}$ and $\boldsymbol{\alpha }\mathbf{,}$ let $\{X_{j}^{\varepsilon }\}$ be
independent with $X_{j}^{\varepsilon }\doteq (X_{1,j}^{\varepsilon
},X_{2,j}^{\varepsilon })\sim \mu ^{\alpha _{j}/\varepsilon }=\mu
_{1}^{\alpha _{j}/\varepsilon }\times \mu _{2}^{\alpha _{j}/\varepsilon }$
for $j\in \left\{ 1,\ldots ,K\right\}$. Define 
\begin{equation*}
\hat{\theta}^{\varepsilon }\left( \mathbf{X}^{\varepsilon };\boldsymbol{%
\alpha }\right) =\sum\nolimits_{\sigma \in \Sigma _{K}}\rho_1 ^{\varepsilon
}\left( \mathbf{X}_{1,\sigma }^{\varepsilon };\boldsymbol{\alpha }\right)
1_{A}(X_{1,\sigma \left( 1\right) }^{\varepsilon },X_{2,\sigma \left(
1\right) }^{\varepsilon })
\end{equation*}%
with the new weights 
\begin{equation*}
\rho_1 ^{\varepsilon }\left( \mathbf{x};\boldsymbol{\alpha }\right) \doteq 
\frac{e^{-\frac{1}{\varepsilon }\sum_{j=1}^{K}\alpha _{j}I_{1}\left(
x_{j}\right) }}{\sum_{\tau \in \Sigma _{K}}e^{-\frac{1}{\varepsilon }%
\sum_{j=1}^{K}\alpha _{j}I_{1}\left( x_{\tau \left( j\right) }\right) }},
\end{equation*}%
where $\mathbf{x}\doteq (x_{1},x_{2},\ldots ,x_{K})$ and with each $x_{j}\in 
\mathcal{X}_{1}$. Note that the weights here depend only on $I_{1}$ instead
of $I.$ One can show that this estimator is also an unbiased estimator of $%
\mu ^{\varepsilon }(A)$, and its decay rate is between that of Monte Carlo
and INS on full set of variables. We will return to this possibility in the
last example of the next section.
\end{remark}

\section{Examples}

\label{sec:examples}

In this section we explore numerical simulations of the INS estimator for
the estimation of rare sets $A$, as specified in Section \ref%
{sec:estimating_rare_events}. Under the conditions of Theorem \ref{Thm2},
the $K$-temperature estimator $\hat{\theta}_{\text{INS}}$ has an asymptotic
decay rate of $(2 - (1/2)^{K-1}) I(A)$ when used to estimate $%
\mu^\varepsilon(A)$ for a suitable set $A$. One can also consider the
preasymptotic (or level $\varepsilon$) decay rate given by $- \varepsilon
\log E (\hat{\theta}_i^\varepsilon)^2$. This quantity can be is estimated
using independent samples. If $\hat{\theta}_1^\varepsilon, \dots, \hat{\theta%
}_N^\varepsilon$ are independent samples of an unbiased estimator of $%
\mu^\varepsilon(A)$, then the \textit{empirical preasymptotic decay rate} is
defined as 
\begin{equation}  \label{eqn:empirical_decay_rate}
-\varepsilon \log \left(\frac{1}{N} \sum_{i=1}^N (\hat{\theta}%
_i^\varepsilon)^2 \right).
\end{equation}

Although the preasymptotic decay rate converges to the true decay rate, it
is only in practice that we get a sense of how small $\varepsilon > 0$ must
be before we see agreement with the limiting decay rate. One expects the
large deviation approximation to be useful for larger values of $\varepsilon$
for some choices of $A$ and to need smaller $\varepsilon$ for others. We
briefly revisit the heuristic asymptotic analysis of Section \ref%
{sec:performance} to elucidate this issue.

First, observe that the choice of $\hat{\theta}_1^\varepsilon, \dots, \hat{%
\theta}_N^\varepsilon$ which maximizes the quantity %
\eqref{eqn:empirical_decay_rate} subject to the constraint $(\hat{\theta}%
_1^\varepsilon + \dots + \hat{\theta}_N^\varepsilon) = N \mu^\varepsilon(A)$
is given by $\hat{\theta}_i = \mu^\varepsilon(A)$. In other words, for a
highly accurate estimator to have a near-optimal decay rate it should
generate values which are centered around the target value $%
\mu^\varepsilon(A)$.

Now consider the case of a two-temperature estimator with independent
samples $X_1^\varepsilon$ and $X_2^\varepsilon$, the second denoting the
higher temperature; the case of a general number of temperatures is
combinatorially more complicated, but the final principle carries over. It
was observed in \eqref{eqn:heuristic1} that when $\hat{\theta}_{\text{INS}%
}^\varepsilon > 0$, we most often have $\hat{\theta}_{\text{INS}%
}^\varepsilon = \rho^\varepsilon(X_2^\varepsilon, X_1^\varepsilon)$. This
weight can be rewritten as 
\begin{equation}  \label{eqn:rewrite_weight}
\rho^\varepsilon(x_2, x_1) = \left[1+\exp\left(\frac{1-\alpha}{\varepsilon}%
(I(x_2) - I(x_1))\right)\right]^{-1}.
\end{equation}
For the estimator to have a good decay rate, whenever $X_1^\varepsilon
\notin A, X_2^\varepsilon \in A$, the weight $\rho^\varepsilon(X_2^%
\varepsilon, X_1^\varepsilon)$ needs to be exponentially small as $%
\varepsilon \rightarrow 0$. It can be seen from \eqref{eqn:rewrite_weight}
that $\rho^\varepsilon(x_2, x_1) < 1/2$ is equivalent to $I(x_2) > I(x_1)$,
which is plausible since we expect the sample in $A$ to have a higher value
of $I$. It is nevertheless possible for small $\varepsilon > 0$ that $%
X_1^\varepsilon \notin A, X_2^\varepsilon \in A$ while $I(X_2^\varepsilon) <
I(X_1^\varepsilon)$. If they occur with high enough probability, such
samples will significantly reduce the preasymptotic decay rate.

Observe that when the target set is of the form $A = \{ x : I(x) \geq c\}$
for some $c >0$, such a mismatch of the weights can never occur since $x_1
\notin A, x_2 \in A$ implies $I(x_2) > I(x_1)$. More generally, sets for
which $\{ x : I(x) \geq I(A)\} \setminus A$ has small probability relative
to $\{x : I(x) \geq I(A)\}$ are in some sense ``well-described'' by the rate
function, and their preasymptotic decay rate trends to the asymptotic value
faster. On the other hand, the probabilities of sets which are not of this
specific form will in general have larger probability, and so having a large
preasymptotic decay rate might not be as important. We will explore these
issues in the examples.


Next we introduce some terminology. The word \textit{sample} will refer to a
single generation of a random variable (which could itself be
multidimensional). A \textit{trial} will be a single realization of the $K$%
-temperature estimator $\hat{\theta }_{\text{INS}}^{\varepsilon}$, which
requires the generation of $K$ samples, one for each temperature. Using $%
N\times K$ samples total we can run $N$ trials, and the collection of these
trials is called a \textit{simulation}. We define the INS estimate $\hat{p}%
_{A}^{\varepsilon}$ of $\mathbb{P}(X^{\varepsilon}\in A)$ by 
\begin{equation*}
\hat{p}_{A}^{\varepsilon}\doteq\frac{1}{N}\sum\nolimits_{i=1}^{N}\hat{\theta 
}^{\varepsilon}(X_{i}^{\varepsilon};\boldsymbol{\alpha}). 
\end{equation*}
The standard error $\hat{\sigma}_{N}$ using $N$ samples is the estimated
standard deviation divided by the square root of the number of samples: 
\begin{equation*}
\hat{\sigma}_{N}\doteq\frac{1}{\sqrt{N}}\sqrt{\frac{1}{N}\sum%
\nolimits_{i=1}^{N}(\hat{\theta}^{\varepsilon}(X_{i}^{\varepsilon};%
\boldsymbol{\alpha})-\hat{p}_{A}^{\varepsilon})^{2}}, 
\end{equation*}
and the relative error $\hat{r}_N$ is the ratio of the estimated standard
deviation and the estimate:%
\begin{equation*}
\hat{r}_N\doteq\frac{\hat{\sigma}_{N}}{\hat{p}_{A}^{\varepsilon}}\sqrt{N}. 
\end{equation*}
The relative error is used to provide formal confidence intervals 
\begin{equation*}
[\hat{p}_A^\varepsilon - 1.96\ \hat{r}_N, \hat{p}_A^\varepsilon + 1.96 \ 
\hat{r}_N]. 
\end{equation*}
Finally, we will report the \textit{normalized empirical decay rate} 
\begin{equation*}
\hat{\mathcal{D}}_N^\varepsilon \doteq \frac{-\varepsilon \log\left(\frac{1}{%
N}\sum_{i=1}^N (\hat{\theta}^\varepsilon(X_i^\varepsilon; \boldsymbol{\alpha}%
)^2\right)}{-\varepsilon \log\left(\frac{1}{N}\sum_{i=1}^N \hat{\theta}%
^\varepsilon(X_i^\varepsilon;\boldsymbol{\alpha})\right)}.
\end{equation*}
All the examples presented below use the asymptotically optimal choice of $%
\boldsymbol{\alpha}$ identified in Lemma \ref{Lem2}.

\begin{example}
The first example is chosen to demonstrate the theoretical properties of the
INS estimator. We consider the case of independent and identically
distributed one-dimensional standard Gaussians, $\xi_{i}\sim N(0,1)$. Set 
\begin{equation*}
X^{\varepsilon}\doteq\varepsilon\sum_{i=1}^{\lfloor1/\varepsilon\rfloor}%
\xi_{i}. 
\end{equation*}
The quantity of interest will be $p_{A}^{\varepsilon}\doteq
P(X^{\varepsilon}\in A)$, where $A$ is the non-convex set $%
A=(-\infty,-0.25]\cup\lbrack0.2,\infty)$. Although simple, owing to the
non-convexity of $A$ this problem requires some care in the design of
importance sampling schemes for estimating $p_{A}^{\varepsilon}$ \cite%
{glawan}. We expect the INS estimator to have a high pre-asymptotic decay
rate, because $\{I(x) \geq I(A)\} = (-\infty, 0.2] \cup [0.2, \infty)$ and
thus $x \notin A, y \in A$ implies $I(y) > I(x)$ with high probability.


In this case there are standard deterministic methods to calculate the
target probabilities up to machine error. The first row in Table \ref%
{table:expected-successes} gives the target temperature and the second row
shows the (rounded) probability of $X^{\varepsilon}$ landing in the target
set $A$. The third row shows the nearest integer amount of expected
successful samples for ordinary Monte Carlo $\mathbb{E}%
\sum_{j=1}^{N}1_{A}(X_{j}^{\varepsilon})=N\mathbb{P}(X^{\varepsilon}\in A)$
when $N=5 \times 10^5$.
\begin{table}[ptbh]
\centering$%
\begin{tabular}{lrrrr}
$\varepsilon$ & $10^{-2}$ & $5\times10^{-3}$ & $2\times10^{-3}$ & $%
1\times10^{-3}$ \\ \hline
$p_{A}^{\varepsilon}$ & $2.90\times 10^{-2}$ & $2.54 \times 10^{-3}$ & $%
3.88\times10^{-6}$ & $1.27\times10^{-10}$ \\ 
$\left[ Np_{A}^{\varepsilon}\right] $ & $14,500$ & $1,250$ & $2$ & $0$%
\end{tabular}
\ $ 
\caption{Probabilities and expected number of successes}
\label{table:expected-successes}
\end{table}

Based on the third row of Table \ref{table:expected-successes}, for $N =
5\times 10^5$ we expect that ordinary Monte Carlo (which corresponds to
using a single temperature) will provide some estimate for $%
\varepsilon=10^{-2}$ and $\varepsilon=5\times10^{-3}$, but not for the other
two values of $\varepsilon$. In Table \ref{table:2} we compare the estimates
obtained from simulations with $K=1,\ldots,5$ temperatures. Confidence
intervals, computation time and other statistics are not reported for this
example. 
\begin{table}[h]
\centering%
\begin{tabular}{lrrrr}
$K\backslash\varepsilon$ & $10^{-2}$ & $5\times10^{-3}$ & $2\times10^{-3}$ & 
$1\times10^{-3}$ \\ \hline
$1$ & $2.89 \times 10^{-2}$ & $2.61\times 10^{-3}$ & $0$ & $0$ \\ 
$2$ & $2.86\times 10^{-2}$ & $2.53\times 10^{-3}$ & $4.01\times10^{-6}$ & $%
5.31\times10^{-10}$ \\ 
$3$ & $2.87\times 10^{-2}$ & $2.61\times 10^{-3}$ & $4.04\times10^{-6}$ & $%
1.47\times10^{-10}$ \\ 
$4$ & $2.90\times 10^{-2}$ & $2.55\times 10^{-3}$ & $3.95\times10^{-6}$ & $%
1.35\times10^{-10}$ \\ 
$5$ & $2.93\times 10^{-2}$ & $2.64\times 10^{-3}$ & $3.78\times10^{-6}$ & $%
1.27\times10^{-10}$%
\end{tabular}
\caption{Estimates of probabilities}
\label{table:2}
\end{table}

For each pair $(\varepsilon, K)$, Table \ref{table:2a} shows the normalized
empirical decay rate $\hat{\mathcal{D}}^\varepsilon_N$.
\begin{table}[h]
\centering%
\begin{tabular}{lrrrr}
$K\backslash\varepsilon$ & $10^{-2}$ & $5\times10^{-3}$ & $2\times10^{-3}$ & 
$1\times10^{-3}$ \\ \hline
$2$ & $1.31$ & $1.31$ & $1.37$ & $1.44$ \\ 
$3$ & $1.46$ & $1.47$ & $1.56$ & $1.59$ \\ 
$4$ & $1.54$ & $1.55$ & $1.64$ & $1.72$ \\ 
$5$ & $1.59$ & $1.62$ & $1.70$ & $1.77$%
\end{tabular}
\caption{Normalized empirical decay rate for the probabilities estimated in
Table \protect\ref{table:2}. No rate is reported for $K=1$ temperature since
it is always equal to 1. }
\label{table:2a}
\end{table}
We see that by combining the estimates from different temperatures the INS
estimator can approximate rare events using far fewer samples than ordinary
Monte Carlo. This is due to the relatively large number of successes from
the higher temperatures, which are properly weighted by $\rho^{\varepsilon
}( \mathbf{x};\boldsymbol{\alpha})$ to produce an unbiased estimate of $%
\mathbb{E}[1_{A}(X^{\varepsilon})]$ for the target temperature $\varepsilon$.

In Figure \ref{fig:weight_distribution} we show the empirical distribution
of $\hat\theta_{\text{INS}}^\varepsilon$ with $K = 5$ temperatures
conditioned on one of the three highest temperatures landing in the target
set; the two lower temperatures have few successes and are therefore
ignored. We take the negative of the base-10 logarithm for easier
readability of the distribution. 
\begin{figure}[tbp]
\centering
\includegraphics[scale=0.25]{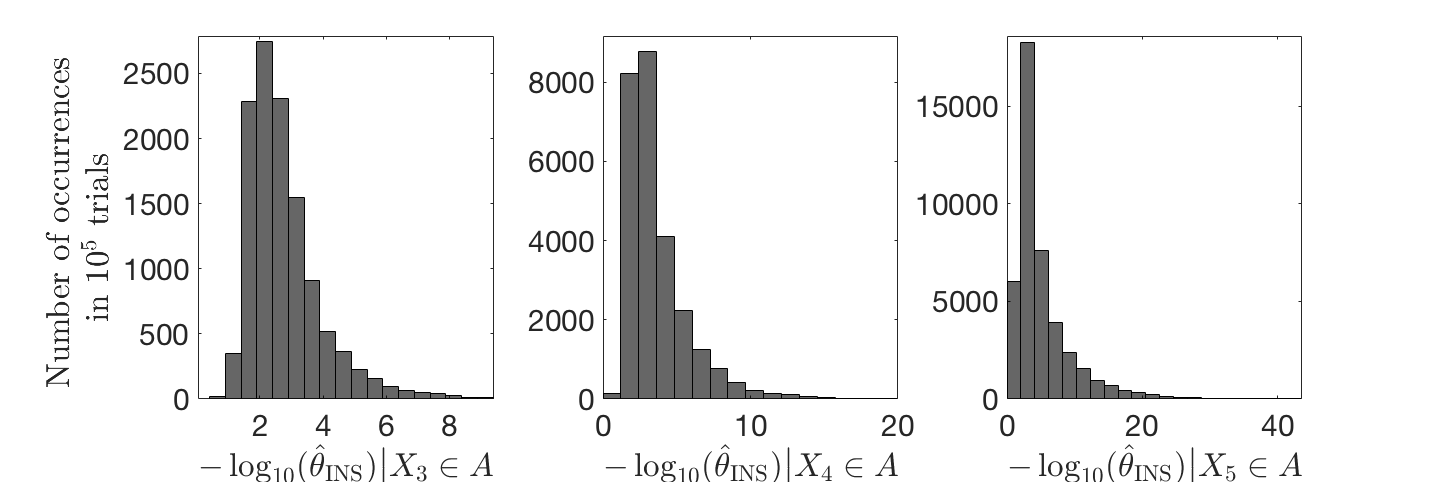}
\caption{The conditional distribution of $\hat{\protect\theta}_{\text{INS}}^%
\protect\varepsilon$ given that the $k^{th}$ temperature landed in $A$ for $%
k = 3, 4, 5$. Here, $\protect\varepsilon = 5\times 10^{-3}$ and $10^5$
samples were used for each of the $K = 5$ temperatures.}
\label{fig:weight_distribution}
\end{figure}
From the right-most histogram in Figure \ref{fig:weight_distribution} one
can estimate that the highest temperature landed in $A$ roughly $4\times 10^4
$ so at least $40\%$ of the INS samples produced were non-zero. One can also
see that the distribution of weights is roughly appropriate, in that the
mode of the histogram times the probability of the $k^{th}$ temperature
landing in the target set is on the same order of magnitude as the
probability $\mu^\varepsilon(A)$.

We can further explore this example by taking the Gaussians to be $d$%
-dimensional while keeping the target set dependent on only the first
coordinate. The rate function becomes the sum of squares $x_i\in\mathbb{R}^d$%
, $I(x_i) = \sum_{j=1}^d x_{i,j}^2/2$. Using $N = 5\times 10^5$ total
samples spread across $K = 5$ temperatures for $\varepsilon = 5\times 10^{-3}
$, we obtain an estimate of $2.49 \times 10^{-3}$, in very close agreement
with the theoretical value. Following up on the discussion at the beginning
of this section, for this problem the superlevel sets of the rate function
provide a relatively poor approximation to the target set. As we see in
Figure \ref{fig:weight_distribution_highdim}, the weights for the low
temperatures tend to be too large, and the weights for the high temperatures
to be too small. This causes a decrease in the preasymptotic decay rate: the
normalized empirical decay rate for the original $d = 1$ was 1.62, while for 
$d = 5$ it was $1.26$. In either case, however, standard Monte Carlo was
outperformed with no problem-specific modification of the algorithm.

\begin{figure}[tbp]
\centering
\includegraphics[scale=0.25]{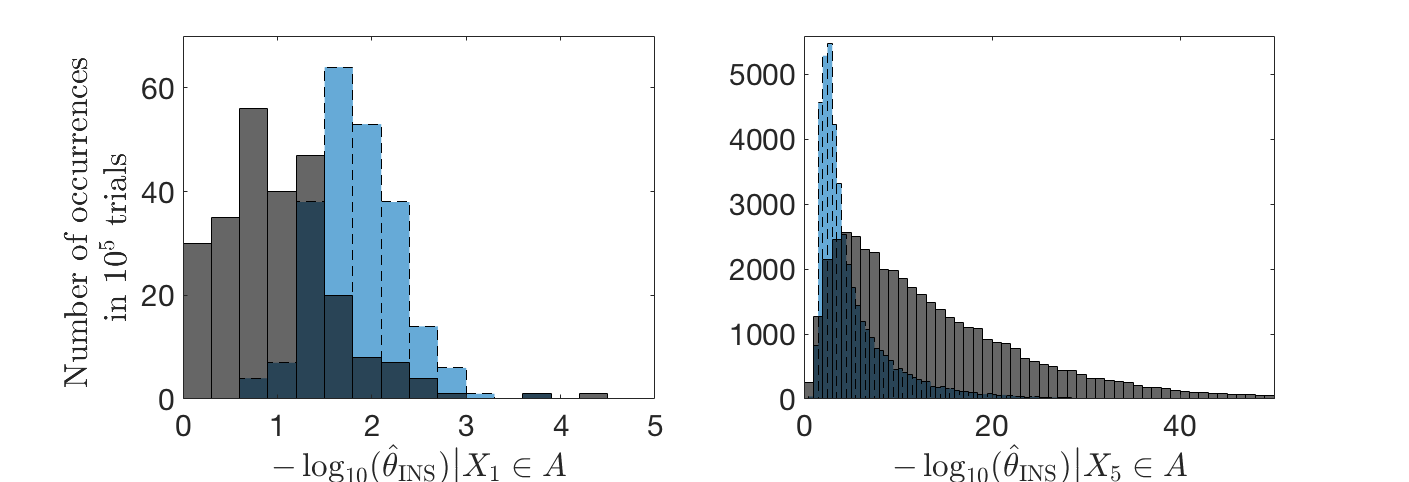}
\caption{Comparison of conditional distributions of INS estimator
conditioned on the lowest and highest temperatures landing in the rare event
set. In blue with dashed edges is the histogram for the one-dimensional
samples, and overlaid in dark gray is the case when 5-dimensional samples
are used. }
\label{fig:weight_distribution_highdim}
\end{figure}
Finally, since it is relatively cheap to generate independent Gaussian
random variables, we can run many simulations for very small $\varepsilon > 0
$ to probe the asymptotic behavior of the estimator. In Figure \ref%
{fig:performance_plot}, we demonstrate that the normalized preasymptotic
decay rates approach $2 - (1/2)^{K-1}$. To avoid the clumping of data points
for small $\varepsilon > 0$, we plot the normalized empirical decay rate
against the inverse temperature $\varepsilon^{-1}$. 
\begin{figure}[tbp]
\includegraphics[scale=0.3]{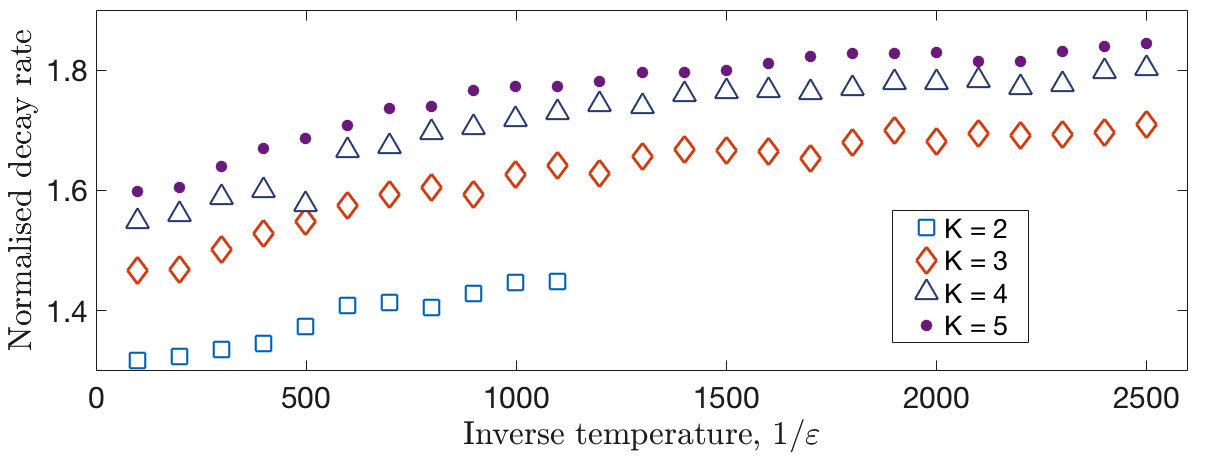}
\caption{The empirical decay rate as a function of the inverse temperature.
Estimates for $K=2$ stop at $1/\protect\varepsilon = 1100$ because it
becomes computationally infeasible to generate enough samples of the second
temperature landing in the target set.}
\label{fig:performance_plot}
\end{figure}
\end{example}

\begin{example}
Define the rate function  
\begin{equation*}
I(x) = (x_1^2 + x_2^2 + x_2 + 1 + x_3^4 + 2x_4^2 + x_4 + 1)^2 - 2.64, 
\end{equation*}
and consider the associated family of probability distributions  
\begin{equation*}
\mu^\varepsilon(dx) = \frac{1}{Z^\varepsilon} e^{-\frac{1}{\varepsilon}
I(x)} dx. 
\end{equation*}
As noted in Example \ref{exa:1}, the sequence $\{\mu^\varepsilon\}$
satisfies an LDP with rate function $I(x)$. Observe that $\mu^\varepsilon$
is unimodal, and that the minimum of $I(x)$ occurs at $x_m = (0, -0.5, 0,
-0.25)$. This $x_m$ is used to determine the shift of $2.64$ to ensure that $%
I(x_m) = 0$ and $I(x) > 0$ for $x \neq x_m$.

We sample from $\mu^\varepsilon$ using MCMC and apply the INS estimator to
the obtained samples without any modification to the INS algorithm. In such
a case where a Markov chain is introduced in order to sample from the
underlying distribution, it would be more natural to use the original
infinite swapping estimator \cite{dupliupladol} directly on the Markov
chain. In the present setting the INS estimator is only asymptotically
unbiased, but for a unimodal distribution that appears to be sufficient for
good performance.

To sample from $\mu^\varepsilon$ for some $\varepsilon > 0$ using MCMC, we
use a Metropolis-Hastings algorithm with proposal distribution 
\begin{equation*}
q^\varepsilon(x, \cdot) \sim x + \varepsilon U v, 
\end{equation*}
where $U \sim \text{Unif}([-1,1])$ and $v$ is drawn uniformly from $\{e_1,
..., e_4\}$, where $e_i$ is the $i^{th}$ basis vector. We sub-sample the
simulated trajectory every $8/\varepsilon$ steps, and verify that the
histograms for the one-dimensional marginal densities look appropriate. The
numerical results are shown in Tables \ref{tbl:A1} and \ref{tbl:A2} and are
based on $5\times 10^4$ trials using $K = 4$ temperatures for each trial.
The computation time does not include the time of generating samples via
MCMC.

We compare two target sets. Define 
\begin{align*}
A_1 &= \{x_1 \leq -0.1, x_2 \geq -0.35, x_3 \geq 0, x_4 \geq -0.2\} \\
A_2 &= \{I(x) \geq 0.5\}.
\end{align*}
The sets are chosen so that the probabilities $\mu^\varepsilon(A_i)$ are on
the same order of magnitude for $\varepsilon = 1/20$, even though $I(A_1) <
I(A_2)$. We expect the decay rate for estimating the probability of landing
in $A_1$ to be lower than that of $A_2$ since the set $A_1$ is an
intersection of sets and constitutes only a fraction of the region $\{x :
I(x) \geq I(A_1)\}$. In contrast, the set $A_2$ is explicitly chosen as a
super-level set of the rate function, so its weights are more likely to be
well-distributed in the sense described in the introduction of this section.
In both cases, however, the empirical decay rate is strictly bigger than $1$.

\begin{table}[tbp]
\begin{tabular}{l|r|r|r}
$\varepsilon$ & $1/20$ & $1/40$ & $1/80$ \\ \hline
LDP Prediction & $8.49 \times 10^{-2}$ & $7.20 \times 10^{-3}$ & $5.20
\times 10^{-5}$ \\ 
Estimate & $4.83 \times 10^{-4}$ & $1.77\times 10^{-5}$ & $3.87\times 10^{-8}
$ \\ 
Confidence interval & $[4.06, 5.61]\times 10^{-4}$ & $[1.03, 2.51]\times
10^{-5}$ & $[1.47, 6.28]\times 10^{-8}$ \\ 
Relative error & $8.19 \times 10^{-2}$ & $2.12 \times 10^{-1}$ & $3.17\times
10^{-1}$ \\ 
Empirical decay rate & $1.24$ & $1.29$ & $1.50$ \\ 
Computation time (s) & 4 & 4 & 4%
\end{tabular}%
\caption{Results for estimating the probability $\protect\mu^\protect%
\varepsilon(A_1)$ for different values of $\protect\varepsilon$ with $N =
5\times 10^4$ trials. We have $I(A_1) \approx 0.123$. }
\label{tbl:A1}
\end{table}
\begin{table}[tbp]
\begin{tabular}{l|r|r|r}
$\varepsilon$ & $1/20$ & $1/40$ & $1/80$ \\ \hline
LDP Prediction & $4.54 \times 10^{-5}$ & $2.06 \times 10^{-9}$ & $4.25
\times 10^{-18}$ \\ 
Estimate & $2.47 \times 10^{-4}$ & $1.94\times 10^{-8}$ & $6.03\times
10^{-17}$ \\ 
Confidence interval & $[2.35, 2.58]\times 10^{-4}$ & $[1.72, 2.15]\times
10^{-8}$ & $[4.78, 7.28]\times 10^{-17}$ \\ 
Relative error & $2.34 \times 10^{-2}$ & $5.73 \times 10^{-2}$ & $1.05\times
10^{-1}$ \\ 
Empirical decay rate & $1.59$ & $1.71$ & $1.83$ \\ 
Computation time (s) & 7 & 5 & 5%
\end{tabular}%
\caption{Results for estimating the probability $\protect\mu^\protect%
\varepsilon(A_2)$ for different values of $\protect\varepsilon$ with $N =
5\times 10^4$ trials. We have $I(A_2) \approx 0.5$.}
\label{tbl:A2}
\end{table}

Observe that the performances reported in Table \ref{tbl:A1} are
significantly lower than the ones in Table \ref{tbl:A2}. When $\varepsilon =
1/20$, the probabilities $\mu^\varepsilon(A_1)$ and $\mu^\varepsilon(A_2)$
are on the same order of magnitude, though the performance for estimating
the latter is markedly higher. This is due to the form of the target set $A_2
$ and is accompanied by a disparity of the probabilities: $%
\mu^\varepsilon(A_2)$ decreases much faster than $\mu^\varepsilon(A_1)$.
Since fewer samples are needed to estimate $\mu^{1/80}(A_1)$ than $%
\mu^{1/80}(A_2)$, a higher decay rate is of less use for the former, which
is reflected in the fact that the relative errors are roughly comparable for
the two sets (though slightly smaller for $A_1$).
\end{example}

\begin{example}
For our final example, we consider the level crossing of a Brownian motion
over the finite time interval $[0, 1]$. Let $W(t), t\in [0,1]$ be a standard
Brownian motion with $W(0) = 0$, $E W(t) = 0$ and $E W(t)^2 = t$. Define the
scaled process 
\begin{equation*}
W^\varepsilon(t) = \sqrt{\varepsilon} W(t), 
\end{equation*}
and let $\mu^\varepsilon$ denote the induced measure on the space of
trajectories $C([0, 1])$ equipped with the topology of uniform convergence.
The sequence of processes $\{W^\varepsilon, \varepsilon > 0\}$ satisfies a
large deviations principle on $C([0, 1])$ with rate function 
\begin{equation*}
I(\phi) = \frac{1}{2}\int_0^1 |\dot{\phi}(s)|^2 ds 
\end{equation*}
if $\phi \in C([0, 1])$ is absolutely continuous and $\phi(0)=0$, and $%
I(\phi) = +\infty$ otherwise. For some fixed $b > 0$, define the rare event
set 
\begin{equation*}
A_b \doteq \{\phi \in C([0, 1]) : \phi(t) \geq b\text{ for some } t \in [0,
1]\}. 
\end{equation*}

One can envision many ways of estimating $\mu^\varepsilon(A_b)$ via
simulation. For instance, for $M \geq 1$ let $\mathcal{G}_M : \mathbb{R}^M
\rightarrow C([0, 1])$ denote the map  
\begin{equation*}
(\mathcal{G}_M(z_1, \dots, z_N))(t) = \frac{1}{\sqrt{M}}\left(\sum_{i=1}^{%
\lfloor M t\rfloor} z_i + (M t - \lfloor N t\rfloor) z_{\lfloor M t\rfloor +
1}\right). 
\end{equation*}
The map $\mathcal{G}_M$ defines a linearly interpolated random walk with
increments $z_i$ over the interval $[i/M - 1/M, i/M]$ for $i = 1, \dots, M$.
If $Z_1, \dots, Z_M$ are i.i.d. normal random variables with zero mean and
unit variance, Donsker's invariance principle ensures that $W_M^\varepsilon
= \sqrt{\varepsilon}\ \mathcal{G}_M(Z_1, \dots, Z_M)$ converges in
distribution to $W^\varepsilon$ as $M \rightarrow \infty$.

The corresponding ``discretised'' rate function on $C([0,1])$ for $%
\{W^\varepsilon_M, \varepsilon > 0\}$ can be given in terms of the rate
function $I_M(z_1, \dots, z_M) = \sum_{j=1}^M z_j^2 / 2$ for ${\varepsilon
(Z_{1},\ldots,Z_{M})}$, and this discretized rate function can in principle
be used to compute the weights $\rho^\varepsilon$ used in the INS estimator.
The problem is that the set $A_b$ is not succinctly described by the
increments $(z_1, \dots, z_M)$. Observe that  
\begin{equation*}
(\sqrt{\varepsilon}\ \mathcal{G})^{-1}_M(A_b) = \left\{(z_1, \dots, z_M) \in 
\mathbb{R}^M : \sum_{i=1}^k z_i \geq \frac{b}{\sqrt{\varepsilon}}\text{ for
some } 1 \leq k \leq M \right\}. 
\end{equation*}
A standard argument via Jensen's inequality shows that the trajectory $\phi
\in A_b$ which minimizes $I(\phi)$ is a straight line between $(0, 0)$ and $%
(1, b)$, which means that with high probability if the trajectory exceeds $b$
then all the random variables $Z_i$ should be close to $b/\sqrt{M \varepsilon%
}$. While $I_M(b/\sqrt{M\varepsilon}, \dots, b/\sqrt{M\varepsilon}) =
b^2/2\varepsilon$ is independent of $M$, $E I_M(Z_1, \dots, Z_M) \approx M/2$%
. This is consistent with the definition of $I(\phi) = \infty$ for $\phi$
which are not absolutely continuous, as is the case for paths of the Wiener
process (almost surely). This means that as $M$ increases for fixed $%
\varepsilon > 0$, there are increasingly many ways for the $Z_i$'s to
produce a large ``cost'' (a large $I_M$ value) without landing in $(\sqrt{%
\varepsilon}\mathcal{G}_M(A_b))^{-1}$. This approximation is not well-suited
for the INS estimator since all the rate function has a ``weak'' dependence
on the individual random variables $Z_i$, with no single one or few largely
responsible for determining occurrence of the rare event.

An alternative approximation is L\'evy's construction of Brownian motion
using the Schauder functions \cite{karshr}. Let $\phi_0(t) = t$ and for $n
\geq 1$, let $I(n)$ denote the set of odd integers between $1$ and $2^n$: $%
I(1) = \{1\}$, $I(2) = \{1, 3\}$, etc. For $n \geq 1$ and $k \in I(n)$,
define  
\begin{equation*}
\phi_{n, k}(t) = 
\begin{cases}
2^{\frac{n-1}{2}}\left(t - \frac{k-1}{2^n}\right) & t \in \left[ \frac{k-1}{%
2^n}, \frac{k}{2^n}\right) \\ 
2^{\frac{n-1}{2}}\left(\frac{k+1}{2^n} - t\right) & t \in \left[\frac{k}{2^n}%
, \frac{k+1}{2^n}\right) \\ 
0 & \text{else}%
\end{cases}%
. 
\end{equation*}
Let $\mathcal{T}$ be a space of infinite triangular arrays, 
\begin{equation*}
\mathcal{T} = \{ z = (z_0, z_{1, 1}, z_{2, 1}, \dots) : z_0, z_{n, k} \in 
\mathbb{R}, n \geq 1, k \in I(n)\}. 
\end{equation*}
For $M \geq 1$, define the map  
\begin{equation*}
(\mathcal{H}_M(z))(t) = z_0 \phi_0(t) + \sum_{n = 1}^M \sum_{k \in I(n)}
z_{n, k} \phi_{n, k}(t). 
\end{equation*}
The preimage of $A_b$ under $\mathcal{H}_M$ can be expressed recursively as
follows. Let $\mathcal{Z}_M = (\mathcal{H}_M)^{-1}(A_b)$. Clearly, $\mathcal{%
Z}_0 = \{z : z_0 \geq b\}$, while for $M \geq 1$,  
\begin{equation*}
\mathcal{Z}_{M+1} = \mathcal{Z}_M \cup \left( \bigcup_{k \in I(M+1)}
\left\{z : (\mathcal{H}_{M+1}(z))\left(\frac{k}{2^{M+1}}\right) \geq
b\right\}\right). 
\end{equation*}
This decomposition can be deduced from the piecewise linearity of $\mathcal{H%
}_M(z)$ and shows that if $\mathcal{H}_M(z) \in A_b$, then $\mathcal{H}%
_{M^{\prime }}(z)$ for all $M^{\prime }\geq M$. Note that $b \phi_0(t)$ is
the most likely escape trajectory, so $z_0$ is by far the most important
contributor to the rare event set. Furthermore, if escape does occur
strictly prior to time $1$, it is likely to happen close to time $1$. For
this reason $z_{n, 2^n-1}$ is more important than, say, $z_{n, 1}$.

As illustrated with the random walk approximation $\mathcal{G}_M$, we want
to avoid including unimportant variables in the rate function. Thus, rather
than simulating completely independent trajectories for each temperature, we
couple the trajectories and use independent copies of only the first four
random variables $Z_0, Z_{1,1}, Z_{2, 1}, Z_{2,3}$. As pointed out in Remark %
\ref{rmk:partial_ins}, the estimator remains unbiased but could suffer a
decrease in the asymptotic decay rate. On the other hand, coupling the
relatively unimportant random variables removes a source of variance at the
preasymptotic level, and at the same time allows control of discretization
error by making $M$ large.

The discretized rate function has 
\begin{equation*}
I_{1}(z) = \frac{z_0^2}{2} + \frac{z_{1,1}^2}{2} + \frac{z_{2,1}^2}{2} + 
\frac{z_{2,3}^2}{2}, 
\end{equation*}
which is all that is needed to construct the INS estimator per Remark \ref%
{rmk:partial_ins}. We present the numerical results in Table \ref{tbl:Haar}. 
\begin{table}[tbp]
\begin{tabular}{l|r|r|r}
$\varepsilon$ & $3 \times 10^{-2}$ & $2 \times 10^{-2}$ & $5 \times 10^{-3}$
\\ \hline
LDP Prediction & $1.55 \times 10^{-2}$ & $1.90 \times 10^{-3}$ & $1.39
\times 10^{-11}$ \\ 
Estimate & $3.95 \times 10^{-3}$ & $6.59\times 10^{-4}$ & $2.77\times
10^{-11}$ \\ 
Confidence interval & $[3.77, 4.12]\times 10^{-3}$ & $[6.03, 7.14]\times
10^{-4}$ & $[0, 6.62]\times 10^{-11}$ \\ 
Relative error & $2.30 \times 10^{-2}$ & $4.30 \times 10^{-2}$ & $7.09\times
10^{-1}$ \\ 
Empirical decay rate & $1.28$ & $1.29$ & $1.55$ \\ 
Computation time (s) & 572 & 567 & 578%
\end{tabular}%
\caption{Results for estimating the probability of crossing the threshold $b
= 0.5$ for different values of $\protect\varepsilon$, with $M = 10$, $N =
5\times 10^5$ and using $K = 3$ temperatures. The LDP estimate is $\exp(-b^2
/ 2\protect\varepsilon)$. The computation time reported includes the time
taken to generate the samples.}
\label{tbl:Haar}
\end{table}
We first comment on the accuracy of the approximation $\mathcal{H}_M$. For
fixed $M \geq 1$, the trajectory $\mathcal{H}_M(z)$ is determined by its
values at points $i/2^M$, $i = 0, \dots, 2^M$. Our numerical experiments
show that for fixed $\varepsilon > 0$, the estimate can noticeably vary with 
$M$, which is an issue inherent to the approximation of a continuous time
process and cannot be resolved by using more temperatures. Nevertheless, the
estimator shows reasonable agreement with the LDP prediction. Furthermore, $%
10^7$ standard Monte Carlo runs of the random walk approximation $\mathcal{G}%
_M$ with $M = 14$ gave estimates of $3.91 \times 10^{-3}$ and $4.34 \times
10^{-4}$ for $\varepsilon = 3\times 10^{-2}$ and $2\times 10^{-2}$,
respectively, with no estimate being produced for $\varepsilon = 5\times
10^{-3}$.

The example demonstrates that the INS estimator can effectively transfer
information carried by the higher temperatures to the lower ones even in an
infinite-dimensional setting, so long as the target set is represented in an
appropriate set of coordinates. An interesting question to explore is
whether an appropriate basis can be constructed for general diffusion
processes.

\end{example}

\section{Relation of INS to other accelerated MC schemes}

\label{sec:relations}

It is useful to compare the qualitative properties of the INS estimator
based on iid samples with other accelerated Monte Carlo schemes. Schemes
that one might consider include standard Monte Carlo, importance sampling
and splitting. Of course in comparison to the INS approach (at least as
developed in this paper) these methods apply to a much broader collection of
problems. Nonetheless, the comparison is useful in highlighting a few
properties of the INS estimator. Our discussion will be brief, and assumes
familiarity with importance sampling and standard forms (e.g., fixed rate)
of splitting.

Suppose one were to consider, for example, the problem of estimating the
probability that a Gaussian random variable $X^{\varepsilon}$ of the form $b+%
\sqrt{\varepsilon}\sigma Z$ falls into a set $A$, where $b$ is a $d$
dimensional vector, $\sigma$ is a $d\times d$ matrix, and $Z$ is a $N(0,I)$
random vector. In order that there be some theoretical bounds on the
variance, we assume that the implementation of both importance sampling and
splitting are applied to a random walk with each of $n$ summands in the
random walk iid $N(b,\sigma\sigma^{T})$ and denoted by $\{Y_{i},i=1,\ldots,n%
\}$, with (for convenience of notation) $\varepsilon=1/n$ and 
\begin{equation}
X_{i}^{n}=\frac{1}{n}\left( Y_{1}+\cdots+Y_{i}\right) \text{ for }%
i=1,\ldots,n.   \label{eqn:RW}
\end{equation}
The problem of interest is thus estimating $P\{X_{n}^{n}\in A\}$. We assume
the boundary of $A$ is regular enough that $\inf_{x\in
A^{\circ}}I(x)=\inf_{x\in\bar{A}}I(x)$. In this case there are both
importance sampling schemes as well as importance functions for the
splitting implementation that are known to be asymptotically optimal, so
long as the boundary of $A$ is regular enough \cite{blaglyled,deadup,dupwan5}%
. Note that in contrast with the INS estimator, interpreting $X_{n}^{n}$ in
terms of the random walk model (\ref{eqn:RW}) is needed when using these
methods to obtain asymptotic optimality. Indeed, the splitting scheme
requires that the splitting event be triggered by the entrance of $X_{i}^{n}$
into a threshold, whose spacing should be proportional to $1/n$. Likewise,
it has been known since \cite{glawan} that for the case on non-convex $A$
one is not able to construct (in general) an asymptotically optimal scheme
that considers only a change of measure applied to $X^{\varepsilon}$ (i.e., $%
X_{n}^{n}$), but rather should consider a dynamic change of measure that
depends on how the simulated trajectory evolves \cite{dupwan3}.

Suppose that we denote estimators obtained using standard Monte Carlo,
importance sampling, splitting and the INS schemes by $\hat{\theta}_{\text{MC%
}}^{n}$, $\hat{\theta}_{\text{IS}}^{n}$, $\hat{\theta}_{\text{Spl}}^{n}$ and 
$\hat{\theta}_{\text{INS}}^{n}$, respectively. In all cases we are
attempting to approximate a probability $\theta^{n}$ by summing independent
realizations, and each estimator satisfies $E\hat{\theta}_{\text{*}}^{n}=$ $%
\theta^{n}$ and hence is unbiased.

In the case of $\hat{\theta}_{\text{MC}}^{n}$, we attempt to approximate $%
\theta^{n}$ as a convex combination of $0$'s and $1$'s. The variance is
approximately the second moment, which decays like the probability itself,
i.e., $\exp-nI(A)$. In this case the relative error is of course bad as $%
n\rightarrow\infty$. When $E\hat{\theta}_{\text{MC}}^{n}=$ $P\{X_{n}^{n}\in
A\}$ is small it is not the $0$'s which contribute to the (relatively) large
variance, but the occasional $1$'s. In comparison with standard Monte Carlo,
all accelerated schemes attempt to cluster the values taken by $\hat{\theta }%
_{\text{*}}^{n}$ (when it is not zero) in a small neighborhood of $%
P\{X_{n}^{n}\in A\}$.

In the case of importance sampling the estimate takes the form of $%
R(Y_{1},\ldots,Y_{n})1_{\{X_{n}^{n}\in A\}}$, where $R(Y_{1},\ldots,Y_{n})$
is the likelihood ratio of the original distribution with respect to the new
distribution used for simulation. For a well designed scheme these
likelihood ratios will cluster around the target value [approximately $%
\exp(-nI(A))$], and for this reason one has good variance reduction and can
achieve small relative errors. However, for a poorly designed scheme \cite%
{glawan} some samples will produce extraordinarily large likelihood ratios
(in fact exponentially large in $n$), and so it is essential that the design
of the simulating distribution be done properly, which in many cases can be
challenging. The INS estimator is similar to IS in that it uses weights
defined by likelihood ratios in the construction. It differs in that
simulations are required for different values of the large deviation
parameter, and in that the weights can never be larger than $1$.

In the case of splitting, one uses a collection of closed sets $C_{0}\subset
C_{1}\subset\cdots\subset C_{J}=%
\mathbb{R}
^{d}$ and splitting rates $R_{j}\in%
\mathbb{N}
,$ $j=0,\ldots,J-1.$ A single particle is started at some point in $%
C_{J}\setminus C_{J-1}$ and evolves according to the law of $\left\{
X_{i}^{n}\right\} .$ When a particle enters a set $C_{j}$ for the first
time, it generates $R_{j}-1$ successors which evolve independently according
to the same law after splitting has occurred. Then a single sample estimator 
$\hat{\theta}_{\text{Spl}}^{n}$ is defined to be $N/\prod_{i=0}^{J-1}R_{i},$
where $N$ in this case is the number of particles simulated which hit $A$ at
time $n.$ For a scheme to work well, it should be designed so that $%
1/\prod_{i=0}^{J-1}R_{i}$ is close to the target value, and at the same time
one would like at least one particle reach $A$ at time $n.$ To design an
effective scheme, one needs to choose the thresholds and splitting rates
carefully. In particular, if the thresholds are too far apart, then the
entire population will die out in the first few generations with very high
probability. On the other hand, if the thresholds are too tight, then the
number of particles will grow exponentially which enhance the computational
effort exponentially. The fact that both $\hat{\theta}_{\text{Spl}}^{n}$ and 
$\hat{\theta}_{\text{INS}}^{n}$ are confined to $[0,1]$ means they are in
some sense \textquotedblleft safer\textquotedblright\ than IS with regard to
variance, though in the case of $\hat{\theta}_{\text{Spl}}^{n}$ the
possibility of exponential growth in the number of particles is in fact as
bad as the possibility of exponentially large likelihood ratios.

Thus INS has features in common with both IS and splitting. However, when
applicable, it does not have the same potential for exponentially bad
behavior that these schemes can exhibit when poorly designed.

\medskip \noindent \textbf{Acknowledgment.} The authors thank an anonymous
referee for a number of helpful comments and also suggestions regarding the
examples. Research supported in part by DOE DE-SC-0010539 and DARPA EQUiPS
W911NF-15-2-0122, NSF DMS-1317199 (PD), and NSERC PGSD3-471576-2015(MS).

\section{Appendix}

While the following result is well known, we were not able to locate a proof
of the multidimensional case, and so include it for completeness.

\begin{lemma}
\label{Lem3}Suppose we are given $h:\mathbb{R}^{d}\rightarrow\mathbb{R}$ that
is continuous and bounded below, and assume it satisfies
\[
\liminf_{M\rightarrow\infty}\inf_{x:\left\Vert x\right\Vert =M}\frac{h\left(
x\right)  }{M}\doteq c>0.
\]
Then
\[
\lim\limits_{\mathbb{\varepsilon}\rightarrow0}\mathbb{\varepsilon}\log
\int_{\mathbb{R}^{d}}e^{-\frac{1}{\varepsilon}h\left(  x\right)  }%
dx=-\inf_{x\in\mathbb{R}^{d}}h\left(  x\right)  .
\]

\end{lemma}

\begin{proof}
Let $\lambda\doteq\inf_{x\in\mathbb{R}^{d}}h\left(  x\right)  >-\infty$. Since
$\mathbb{\varepsilon}\log\int_{\mathbb{R}^{d}}e^{-\frac{1}{\varepsilon
}[h\left(  x\right)  +\lambda]}dx=$ $\mathbb{\varepsilon}\log\int
_{\mathbb{R}^{d}}e^{-\frac{1}{\varepsilon}h\left(  x\right)  }dx-\lambda$, we
can assume without loss in proving the lemma that $\lambda=0$. Choose
$M<\infty$ such that $h(x)\geq c\left\Vert x\right\Vert /2$ for $\left\Vert
x\right\Vert \geq M$. Then
\[
  \limsup_{\mathbb{\varepsilon}\rightarrow0}\mathbb{\varepsilon}\log
\int_{\mathbb{R}^{d}}e^{-\frac{1}{\varepsilon}h\left(  x\right)  }dx
\leq\limsup_{\mathbb{\varepsilon}\rightarrow0}\mathbb{\varepsilon}%
\log\left(  \int_{\mathbb{R}^{d}}1_{\{x:\left\Vert x\right\Vert \geq
M\}}e^{-\frac{1}{2\varepsilon}c\left\Vert x\right\Vert }dx+\int_{\mathbb{R}%
^{d}}1_{\{x:\left\Vert x\right\Vert \leq M\}}dx\right)
\]
Changing the coordinates to polar coordinates in the first integral gives
\[
\int_{\mathbb{R}^{d}}1_{\{x:\left\Vert x\right\Vert \geq M\}}e^{-\frac
{1}{2\varepsilon}c\left\Vert x\right\Vert }dx=C_{1}\varepsilon e^{-\frac
{1}{2\varepsilon}cM}+C_{2}\varepsilon^{2}e^{-\frac{1}{2\varepsilon}cM}%
\leq\tilde{C}%
\]
for some positive constant $\tilde{C}$ and for all $\varepsilon\leq1,$ and thus
\[
\limsup_{\mathbb{\varepsilon}\rightarrow0}\mathbb{\varepsilon}\log
\int_{\mathbb{R}^{d}}e^{-\frac{1}{\varepsilon}h\left(  x\right)  }%
dx\leq\limsup_{\mathbb{\varepsilon}\rightarrow0}\mathbb{\varepsilon}%
\log\left(  \tilde{C}+\text{Vol}(M)\right)  =0,
\]
where Vol$(M)$ is the volume of the ball of radius $M$ in $\mathbb{R}^{d}$.

To prove the reverse inequality, let $x^{\ast}$ be such that $|h\left(
x^{\ast}\right)  -\lambda|<\delta/2,$ then by continuity of $h$, given
$\delta>0$ choose $m>0$ such that $|h(x)-h\left(  x^{\ast}\right)  |\leq
\delta/2$ for if $\left\Vert x-x^{\ast}\right\Vert \leq m$. Then
\[
\liminf_{\mathbb{\varepsilon}\rightarrow0}\mathbb{\varepsilon}\log
\int_{\mathbb{R}^{d}}e^{-\frac{1}{\varepsilon}h\left(  x\right)  }dx  
\geq\liminf_{\mathbb{\varepsilon}\rightarrow0}\mathbb{\varepsilon}\log\left[
e^{-\frac{1}{\varepsilon}\delta}\text{Vol}(m)\right] 
  =-\delta.
\]
Since $\delta>0$ is arbitrary, the reverse bound also holds.
\end{proof}

\begin{theorem}
\label{thm:ldp}Suppose $X^{\varepsilon}$ has distribution $\mu^{\varepsilon},
$ where $\mu^{\varepsilon}$ is a probability measure on $\mathbb{R}^{d}$ of
Gibbs form, i.e., 
\begin{equation*}
\mu^{\varepsilon}\left( A\right) =\frac{\int_{A}e^{-\frac{1}{\varepsilon }%
I\left( x\right) }dx}{Z^{\varepsilon}}\text{ for }A\in\mathcal{B}(\mathbb{R}%
^{d}), 
\end{equation*}
with normalizing constant $Z^{\varepsilon}=\int_{\mathbb{R}^{d}}e^{-\frac {1%
}{\varepsilon}I\left( x\right) }dx$. Moreover, assume that $I:\mathbb{R}%
^{d}\rightarrow\lbrack0,\infty)$ is a continuous rate function with 
\begin{equation*}
\liminf_{M\rightarrow\infty}\inf_{x:\left\Vert x\right\Vert =M}\frac{I\left(
x\right) }{M}>0. 
\end{equation*}
Then $\left\{ X^{\varepsilon}\right\} $ satisfies a large deviation
principle with rate $I.$
\end{theorem}

\begin{proof}
It suffices to prove $\left\{  X^{\varepsilon}\right\}  $ satisfies the
Laplace principle with rate $I$ \cite[Section 1.2]{dupell4}. Thus we consider
$f:\mathbb{R}^{d}\rightarrow\mathbb{R}$ that is bounded and continuous. We can
write
\[
\mathbb{\varepsilon}\log Ee^{-\frac{1}{\varepsilon}f\left(  X^{\varepsilon
}\right)  }    =\mathbb{\varepsilon}\log\int_{\mathbb{R}^{d}}e^{-\frac
{1}{\varepsilon}f\left(  x\right)  }\mu^{\varepsilon}\left(  dx\right) 
  =\mathbb{\varepsilon}\log\int_{\mathbb{R}^{d}}e^{-\frac{1}{\varepsilon
}\left(  f\left(  x\right)  +I\left(  x\right)  \right)  }%
dx-\mathbb{\varepsilon}\log\int_{\mathbb{R}^{d}}e^{-\frac{1}{\varepsilon
}I\left(  x\right)  }dx.
\]
By applying Lemma \ref{Lem3} to $f+I$ and $I,$ we get
\begin{align*}
\lim_{\mathbb{\varepsilon}\rightarrow0}\mathbb{\varepsilon}\log Ee^{-\frac
{1}{\varepsilon}f\left(  X^{\varepsilon}\right)  }  &  =\lim
_{\mathbb{\varepsilon}\rightarrow0}\mathbb{\varepsilon}\log\int_{\mathbb{R}%
^{d}}e^{-\frac{1}{\varepsilon}\left(  f\left(  x\right)  +I\left(  x\right)
\right)  }dx-\lim_{\mathbb{\varepsilon}\rightarrow0}\mathbb{\varepsilon}%
\log\int_{\mathbb{R}^{d}}e^{-\frac{1}{\varepsilon}I\left(  x\right)  }dx\\
&  =-\inf_{x\in\mathbb{R}^{d}}\left[  f\left(  x\right)  +I\left(  x\right)
\right]  +\inf_{x\in\mathbb{R}^{d}}I\left(  x\right) \\
&  =-\inf_{x\in\mathbb{R}^{d}}\left[  f\left(  x\right)  +I\left(  x\right)
\right]  ,
\end{align*}
and the conclusion follows.
\end{proof}

\begin{proof}
[Proof of Lemma \ref{Lem5}]To show that $\{\tilde{X}^{\varepsilon}\}$
satisfies the LDP on $\mathcal{\tilde{X}}$ with rate function $\tilde{I}$, we
prove the lower bound first and then the upper bound. Given any $x\in
\mathcal{\tilde{X}}$ $\ $and $\delta\in(0,\infty)$, we know that
\[
\tilde{B}\left(  x,\delta\right)  \dot{=}\left\{  y\in\mathcal{\tilde{X}%
}:d\left(  x,y\right)  <\delta\right\}  \subset B\left(  x,\delta\right)
\subset\tilde{B}\left(  x,\delta\right)  \cup\left\{  I=\infty\right\}  .
\]
Since $\left\{  X^{\varepsilon}\right\}  \subset\mathcal{X}$ satisfies the LDP
with rate function $I$,
\[
\liminf\limits_{\varepsilon\rightarrow0}\varepsilon\log P\left(
X^{\varepsilon}\in B\left(  x,\delta\right)  \right)  \geq-\inf_{z\in B\left(
x,\delta\right)  }I\left(  z\right)  =-\inf_{z\in\tilde{B}\left(
x,\delta\right)  }I\left(  z\right)  ,
\]
where the equality holds since $\inf_{z\in B\left(  x,\delta\right)  }I\left(
z\right)  $ is finite. We also have
\[
P\left(  X^{\varepsilon}\in B\left(  x,\delta\right)  \right)    \leq
P(X^{\varepsilon}\in\tilde{B}\left(  x,\delta\right)  )+P\left(
X^{\varepsilon}\in\left\{  I=\infty\right\}  \right) 
  =P(\tilde{X}^{\varepsilon}\in\tilde{B}\left(  x,\delta\right)  )+P\left(
X^{\varepsilon}\in\left\{  I=\infty\right\}  \right)  ,
\]
which implies%
\begin{align*}
 \liminf\limits_{\varepsilon\rightarrow0}\varepsilon\log P\left(
X^{\varepsilon}\in B\left(  x,\delta\right)  \right) 
&  \leq\liminf\limits_{\varepsilon\rightarrow0}\varepsilon\log\left[
2\max\left\{  P(\tilde{X}^{\varepsilon}\in\tilde{B}\left(  x,\delta\right)
),P\left(  X^{\varepsilon}\in\left\{  I=\infty\right\}  \right)  \right\}
\right] \\
&  =\max\left\{  \liminf\limits_{\varepsilon\rightarrow0}\varepsilon\log
P(\tilde{X}^{\varepsilon}\in\tilde{B}\left(  x,\delta\right)  ),\liminf
\limits_{\varepsilon\rightarrow0}\varepsilon\log P\left(  X^{\varepsilon}%
\in\left\{  I=\infty\right\}  \right)  \right\} \\
&  =\liminf\limits_{\varepsilon\rightarrow0}\varepsilon\log P(\tilde
{X}^{\varepsilon}\in\tilde{B}\left(  x,\delta\right)  ).
\end{align*}
Therefore%
\[
\liminf\limits_{\varepsilon\rightarrow0}\varepsilon\log P(\tilde
{X}^{\varepsilon}\in\tilde{B}\left(  x,\delta\right)  )\geq-\inf_{z\in
\tilde{B}\left(  x,\delta\right)  }I\left(  z\right)  ,
\]
which is easily seen to imply the large deviation lower bound for arbitrary open sets.

Next we prove the large deviation upper bound. Given any closed set $F\subset
\mathcal{\tilde{X}}$, let $\bar{F}$ be the closure of $F$ in $\mathcal{X}$,
and observe that $F\subset\bar{F}\subset F\cup\{I=\infty\}$. Thus
\begin{align*}
  \limsup\limits_{\varepsilon\rightarrow0}\varepsilon\log P\left(
X^{\varepsilon}\in\bar{F}\right) 
&  \leq\max\left\{  \limsup\limits_{\varepsilon\rightarrow0}%
\varepsilon\log P(\tilde{X}^{\varepsilon}\in F),\limsup\limits_{\varepsilon
\rightarrow0}\varepsilon\log P\left(  X^{\varepsilon}\in\left\{
I=\infty\right\}  \right)  \right\} \\
&  =\limsup\limits_{\varepsilon\rightarrow0}\varepsilon\log P(\tilde
{X}^{\varepsilon}\in F)\\
&  =\limsup\limits_{\varepsilon\rightarrow0}\varepsilon\log P\left(
X^{\varepsilon}\in F\right) \\
&  \leq\limsup\limits_{\varepsilon\rightarrow0}\varepsilon\log P\left(
X^{\varepsilon}\in\bar{F}\right)  .
\end{align*}
Therefore by the large deviation upper bound for $\left\{  X^{\varepsilon}\right\}  $%
\[
\limsup\limits_{\varepsilon\rightarrow0}\varepsilon\log P(\tilde
{X}^{\varepsilon}\in F)=\limsup\limits_{\varepsilon\rightarrow0}%
\varepsilon\log P\left(  X^{\varepsilon}\in\bar{F}\right)  \leq-\inf_{z\in
\bar{F}}I\left(  z\right)  =-\inf_{z\in F}I\left(  z\right)  ,
\]
which completes the proof of the large deviation upper bound.
\end{proof}

\begin{proof}
[Proof of Lemma \ref{Lem1}]We first consider part 1. For a closed set
$B\subset$ $\mathcal{X}$ and $M\in%
\mathbb{N}
,$ define
\[
g_{M}\left(  x\right)  \doteq1_{B}\left(  x\right)  +M\cdot d\left(
x,B_{M}^{c}\right)  1_{B_{M}\backslash B}\left(  x\right)  ,
\]
where
\[
B_{M}\doteq\left\{  x\in\mathcal{X}\text{ : }d\left(  x,B\right)  \leq\frac
{1}{M}\right\}  .
\]
Note that $B_{M}$ is closed, $B_{M}\downarrow B$, and $g_{M}$ is a bounded
continuous function with $0\leq g_{M}\leq1$. Since for any
$x,\mathbb{\varepsilon}$, and $M$%
\[
1_{B}\left(  x\right)  =e^{-\infty1_{B^{c}}\left(  x\right)  }\leq
e^{-\frac{1}{\varepsilon}M\left(  1-g_{M}\left(  x\right)  \right)  },
\]
we have
\[
-\varepsilon\log E\left(  e^{-\frac{1}{\varepsilon}f\left(  X^{\varepsilon
}\right)  }1_{B}\left(  X^{\varepsilon}\right)  \right)  \geq-\varepsilon\log
E\left(  e^{-\frac{1}{\varepsilon}\left[  f\left(  X^{\varepsilon}\right)
+M\left(  1-g_{M}\left(  X^{\varepsilon}\right)  \right)  \right]  }\right)
.
\]
Moreover, since $f$ is bounded below and lower semi-continuous, so is
$f\left(  x\right)  +M\left(  1-g_{M}\left(  x\right)  \right)  $, and thus by
the Laplace principle upper bound \cite[Corollary 1.2.5]{dupell4}
\[
\liminf_{\varepsilon\rightarrow0}-\varepsilon\log E\left(  e^{-\frac
{1}{\varepsilon}\left[  f\left(  X^{\varepsilon}\right)  +M\left(
1-g_{M}\left(  X^{\varepsilon}\right)  \right)  \right]  }\right)  \geq
\inf_{x\in\mathcal{X}}\left[  f\left(  x\right)  +M\left(  1-g_{M}\left(
x\right)  \right)  +I\left(  x\right)  \right]  .
\]

Let
\[
C_{M}\doteq\inf\limits_{x\in\mathcal{X}}\left[  f\left(  x\right)  +M\left(
1-g_{M}\left(  x\right)  \right)  +I\left(  x\right)  \right]  \text{ and
}C_{\infty}\doteq\inf\limits_{x\in B}\left[  f\left(  x\right)  +I\left(
x\right)  \right]  .
\]
Since $I$ is a rate function $I(x)=0$ for some $x$, and thus $C_{\infty
}<\infty$. If we can find a subsequence $\left\{  M_{j}\right\}
_{j\in\mathbb{N}}$ such that
\[
\lim_{j\rightarrow\infty}C_{M_{j}}\geq C_{\infty},
\]
then we are done. We can assume $\sup_{M}C_{M}<\infty$, since otherwise there
is nothing to prove. Let $b$ be a finite lower bound for $f$. Then for any
$M>C_{\infty}-b$ we can write $C_{M}=\min\{C_{M}^{1},C_{M}^{2}\}$, where
\begin{align*}
C_{M}^{1}  &  \doteq\inf\limits_{x\in B_{M}}\left[  f\left(  x\right)
+M\left(  1-g_{M}\left(  x\right)  \right)  +I\left(  x\right)  \right] \\
C_{M}^{2}  &  \doteq\inf\limits_{x\in B_{M}^{c}}\left[  f\left(  x\right)
+M\left(  1-g_{M}\left(  x\right)  \right)  +I\left(  x\right)  \right]  .
\end{align*}
Since
\[
C_{M}^{2}   =\inf\limits_{x\in B_{M}^{c}}\left[  f\left(  x\right)
+I\left(  x\right)  \right]  +M
  >\inf\limits_{x\in B_{M}^{c}}\left[  f\left(  x\right)  +I\left(  x\right)
\right]  +C_{\infty}-b
  \geq C_{\infty}
  =\inf\limits_{x\in B}\left[  f\left(  x\right)  +I\left(  x\right)
\right]
\]
and
\[
C_{M}^{1}\leq\inf\limits_{x\in B}\left[  f\left(  x\right)  +M\left(
1-g_{M}\left(  x\right)  \right)  +I\left(  x\right)  \right]  =C_{\infty},
\]
we have
\[
C_{M}=\inf\limits_{x\in B_{M}}\left[  f\left(  x\right)  +M\left(
1-g_{M}\left(  x\right)  \right)  +I\left(  x\right)  \right]  .
\]

Let $\left\{  x_{M}\right\}  \subset B_{M}$ come within $1/M$ in the
definition of $C_{M}$, so that
\[
f\left(  x_{M}\right)  +M\left(  1-g_{M}\left(  x_{M}\right)  \right)
+I\left(  x_{M}\right)  \leq C_{M}+\frac{1}{M}.
\]
Since $f$ is bounded below and $1-g_{M}\geq0,$ we have $\sup_{M}I\left(
x_{M}\right)  \leq\sup_{M}C_{M}+1-b<\infty.$ Then because $I$ has compact
level sets, $\left\{  x_{M}\right\}  $ has a subsequence $\left\{  x_{M_{j}%
}\right\}  $ converging to $x^{\ast}\in B\dot{=}\cap_{M}B_{M}.$ Hence,%
\begin{align*}
\liminf_{j\rightarrow\infty}C_{M_{j}}  &  \geq\liminf_{j\rightarrow\infty
}\left[  f\left(  x_{M_{j}}\right)  +M\left(  1-g_{M_{j}}\left(  x_{M_{j}%
}\right)  \right)  +I\left(  x_{M_{j}}\right)  \right] \\
&  \geq\liminf_{j\rightarrow\infty}\left[  f\left(  x_{M_{j}}\right)
+I\left(  x_{M_{j}}\right)  \right] \\
&  \geq f\left(  x^{\ast}\right)  +I\left(  x^{\ast}\right) \\
&  \geq C_{\infty},
\end{align*}
where the second inequality comes from the definition of $g_{M}$, and the
third inequality is due to the lower semicontinuity of $f$ and $I$. This
completes the proof of part 1.

Turning to part 2, observe that
\[
\sum_{\ell=1}^{N}e^{-\frac{1}{\varepsilon}f_{\ell}\left(  X^{\varepsilon
}\right)  }\geq\max_{\ell\in\left\{  1,\ldots,N\right\}  }\left\{
e^{-\frac{1}{\varepsilon}f_{\ell}\left(  X^{\varepsilon}\right)  }\right\}
=e^{-\frac{1}{\varepsilon}\min_{\ell\in\left\{  1,\ldots,N\right\}  }\left\{
f_{\ell}(X^{\varepsilon})\right\}  }.
\]
Thus%
\begin{align*}
&  \liminf_{\varepsilon\rightarrow0}-\varepsilon\log E\left(  \frac
{e^{-\frac{1}{\varepsilon}\left(  g_{1}\left(  X^{\varepsilon}\right)
+g_{1}\left(  X^{\varepsilon}\right)  \right)  }}{\left(  \sum_{\ell=1}%
^{N}e^{-\frac{1}{\varepsilon}f_{\ell}\left(  X^{\varepsilon}\right)  }\right)
^{2}}1_{B_{1}}\left(  X^{\varepsilon}\right)  1_{B_{2}}\left(  X^{\varepsilon
}\right)  \right) \\
&  \quad\geq\liminf_{\varepsilon\rightarrow0}-\varepsilon\log E\left(
\frac{e^{-\frac{1}{\varepsilon}\left(  g_{1}\left(  X^{\varepsilon}\right)
+g_{2}\left(  X^{\varepsilon}\right)  \right)  }}{e^{-2\frac{1}{\varepsilon
}\min_{\ell\in\left\{  1,\ldots,N\right\}  }\left\{  f_{\ell}(X^{\varepsilon
})\right\}  }}1_{B_{1}}\left(  X^{\varepsilon}\right)  1_{B_{2}}\left(
X^{\varepsilon}\right)  \right)  .
\end{align*}
Then by the Cauchy-Schwarz inequality,
\begin{align*}
&  \liminf_{\varepsilon\rightarrow0}-\varepsilon\log E\left(  \frac
{e^{-\frac{1}{\varepsilon}\left(  g_{1}\left(  X^{\varepsilon}\right)
+g_{2}\left(  X^{\varepsilon}\right)  \right)  }}{e^{-2\frac{1}{\varepsilon
}\min_{\ell\in\left\{  1,\ldots,N\right\}  }\left\{  f_{\ell}(X^{\varepsilon
})\right\}  }}1_{B_{1}}\left(  X^{\varepsilon}\right)  1_{B_{2}}\left(
X^{\varepsilon}\right)  \right) \\
&  \quad\geq\frac{1}{2}\liminf_{\varepsilon\rightarrow0}-\varepsilon\log
E\left(  e^{-\frac{1}{\varepsilon}\left(  2g_{1}\left(  X^{\varepsilon
}\right)  -2\min_{\ell\in\left\{  1,\ldots,N\right\}  }\left\{  f_{\ell
}(X^{\varepsilon})\right\}  \right)  }1_{B_{1}}\left(  X^{\varepsilon}\right)
\right) \\
&  \quad\quad+\frac{1}{2}\liminf_{\varepsilon\rightarrow0}-\varepsilon\log
E\left(  e^{-\frac{1}{\varepsilon}\left(  2g_{2}\left(  X^{\varepsilon
}\right)  -2\min_{\ell\in\left\{  1,\ldots,N\right\}  }\left\{  f_{\ell
}(X^{\varepsilon})\right\}  \right)  }1_{B_{2}}\left(  X^{\varepsilon}\right)
\right)  .
\end{align*}
We can now we apply part 1 of the lemma with $B\doteq B_{r}$ and $f\left(
x\right)  =2g_{r}\left(  x\right)  -2\min_{\ell\in\left\{  1,\ldots,N\right\}
}\left\{  f_{\ell}(x)\right\}  ,$ $r=1,2.$ Obviously $B$ is closed and $f$ is
lower semi-continuous and bounded below, so the conditions of part 1 apply.
Therefore%
\begin{align*}
&  \liminf_{\varepsilon\rightarrow0}-\varepsilon\log E\left(  \frac
{e^{-\frac{1}{\varepsilon}\left(  g_{1}\left(  X^{\varepsilon}\right)
+g_{2}\left(  X^{\varepsilon}\right)  \right)  }}{\left(  \sum_{\ell=1}%
^{N}e^{-\frac{1}{\varepsilon}f_{\ell}\left(  X^{\varepsilon}\right)  }\right)
^{2}}1_{B_{1}}\left(  X^{\varepsilon}\right)  1_{B_{2}}\left(  X^{\varepsilon
}\right)  \right) \\
&  \quad\geq\frac{1}{2}\inf_{x\in B_{1}}\left[  2g_{1}(x)+I(x)-2\min_{\ell
}\left\{  f_{\ell}(x)\right\}  \right] +\frac{1}{2}\inf_{x\in B_{2}}\left[  2g_{2}(x)+I(x)-2\min_{\ell
}\left\{  f_{\ell}(x)\right\}  \right] \\
&  \quad\geq\min_{r\in\left\{  1,2\right\}  }\left\{  \inf_{x\in B_{r}}\left[
2g_{r}(x)+I(x)-2\min_{\ell}\left\{  f_{\ell}(x)\right\}  \right]  \right\}  .
\end{align*}

For part 3, if $B\subset\left\{  I=\infty\right\}  ,$ then there is nothing to
prove. If $B\cap\left\{  I<\infty\right\}  \neq\emptyset$, then $\inf_{x\in
B}\left[  f\left(  x\right)  +I\left(  x\right)  \right]  <\infty$. Since $f$
and $I$ are both bounded below, $f$ is continuous, $I$ is continuous on the
domain of finiteness and $B$ is the closure of its interior, for any $\nu>0$
there exists $x_{\nu}$ in the interior of $B$ and $\delta>0$ such that%
\[
f\left(  x_{\nu}\right)  +I\left(  x_{\nu}\right)  \leq\inf_{x\in B}\left[
f\left(  x\right)  +I\left(  x\right)  \right]  +\nu,
\]
and also $\left\vert f\left(  x_{\nu}\right)  -f\left(  y\right)  \right\vert
<\nu$ whenever $y$ lies in the open ball $B\left(  x_{\nu},\delta\right)
\doteq\left\{  y\in\mathcal{X}:d\left(  y,x_{\nu}\right)  <\delta\right\}  .$
We can choose $\delta>0$ small enough such that $B\left(  x_{\nu}%
,\delta\right)  \subset B$. This implies
\[
E\left[  e^{-\frac{1}{\varepsilon}f\left(  X^{\varepsilon}\right)  }%
1_{B}\left(  X^{\varepsilon}\right)  \right]     \geq E\left[  e^{-\frac
{1}{\varepsilon}f\left(  X^{\varepsilon}\right)  }1_{B\cap B\left(  x_{\nu
},\delta\right)  }\left(  X^{\varepsilon}\right)  \right] 
  \geq e^{-\frac{1}{\varepsilon}\left(  f\left(  x_{\nu}\right)  +\nu\right)
}P\left(  X^{\varepsilon}\in B\left(  x_{\nu},\delta\right)  \right)  .
\]
Then by the large deviation principle for $\left\{  X^{\varepsilon}\right\}  $
with open set $B\left(  x_{\nu},\delta\right)  $,
\[
\limsup_{\varepsilon\rightarrow0}-\varepsilon\log P\left(  X^{\varepsilon}\in
B\left(  x_{\nu},\delta\right)  \right)  \leq I\left(  B\left(  x_{\nu}%
,\delta\right)  \right)  .
\]
Therefore
\begin{align*}
\limsup_{\varepsilon\rightarrow0}-\varepsilon\log E\left[  e^{-\frac
{1}{\varepsilon}f\left(  X^{\varepsilon}\right)  }1_{B\left(  x_{\nu}%
,\delta\right)  }\left(  X^{\varepsilon}\right)  \right] 
 &  \leq f\left(  x_{\nu}\right)  +\nu+\limsup_{\varepsilon
\rightarrow0}-\varepsilon\log P\left(  X^{\varepsilon}\in B\left(  x_{\nu
},\delta\right)  \right) \\
 &  \leq f\left(  x_{\nu}\right)  +\nu+I\left(  B\left(  x_{\nu
},\delta\right)  \right) \\
 &  \leq f\left(  x_{\nu}\right)  +\nu+I\left(  x_{\nu}\right) \\
&  \leq\inf_{x\in B}\left[  f\left(  x\right)  +I\left(  x\right)
\right]  +\nu.
\end{align*}
Sending $\nu\rightarrow0$ completes the proof.
\end{proof}

\bibliographystyle{plain}
\bibliography{main}

\end{document}